\documentclass[reqno,12pt]{amsart}

\usepackage{graphicx,pdfsync,amssymb, latexsym,amsfonts,amsbsy,color,subcaption,esint,mathtools,enumerate,nicefrac,bm}
\usepackage{hyperref}
\usepackage{booktabs}
\usepackage{multirow}
\usepackage{appendix}
\hypersetup{
colorlinks=true,
linkcolor=blue,
filecolor=blue,
urlcolor=blue,
citecolor=blue,
}

\usepackage{float}
\usepackage{mathrsfs}
\restylefloat{table}

\usepackage{enumitem}
\setitemize{itemsep=0.4em}
\setenumerate{itemsep=0.5em}

\usepackage[utf8]{inputenc}
\usepackage[T1]{fontenc}
\usepackage[top=1.1in, bottom=1in, left=1in, right=1in]{geometry}

\DeclareMathOperator{\curl}{curl}

\newtheorem{theorem}{Theorem}[section]
\newtheorem{lemma}[theorem]{Lemma}
\newtheorem{proposition}[theorem]{Proposition}
\newtheorem{definition}[theorem]{Definition}

\newtheorem{remark}[theorem]{Remark}

\newcommand{\thistheoremname}{}
\newtheorem{genericthm}[theorem]{\thistheoremname}

 \newtheorem*{genericthm*}{\thistheoremname}
\newenvironment{namedthm*}[1]
  {\renewcommand{\thistheoremname}{#1}%
   \begin{genericthm*}}
  {\end{genericthm*}}

\usepackage{times}
\usepackage{setspace}
\newcommand{\dd}{\mathop{}\!\mathrm{d}}
\let\del\partial

\newcommand{\R}{\mathbb R}
\newcommand{\QQ}{\mathbb Q}
\newcommand{\ZZ}{\mathbb Z}

\newcommand{\BMO}{{\rm BMO}}

\newcommand{\uin}[0]{u^{\textup{in}}}

\newcommand{\bin}[0]{B^{\textup{in}}}

\newcommand{\PP}[0]{\mathbb{P}_{\neq 0}}

\newcommand{\TT}[0]{\mathsf{T}}
\newcommand{\TTT}[0]{\mathbb{T}}

\newcommand{\wtq}[1][q+1]{w^{\textup{(mhd)}}_{2q}}
\newcommand{\wpq}[1][q+1]{w^{\textup{(pri)}}_{2q}}
\newcommand{\wpp}{ {w}^{\text{(h)}} }
\newcommand{\wpm}{ {w}^{\text{(h,m)}} }
\newcommand{\wpr}{ {w}^{\text{(h,r)}} }
\newcommand{\wss}{ {w}^{\text{(i)}} }

\newcommand{\wmhd}{ {w}^{\text{(m)}} }
\newcommand{\dmhd}{ {d}^{\text{(m)}} }

\newcommand{\ii}{\textup{i}}

\newcommand{\Div}{\text{div}}

\newcommand{\LinfB}[1][q+1]{\widetilde{L}^\infty_T B^{\frac{1}{2}}_{2,1}}
\newcommand{\LoB}[1][q+1]{{L}^1_t B^{1/2}_{2,1}}
\newcommand{\LoBt}[1][q+1]{\widetilde {L}^1_t B^{5/2}_{2,1}}
\newcommand{\spt}{\text{spt}}
\newcommand{\kp}{{k}^{\perp}}

\allowdisplaybreaks[4]

\numberwithin{equation}{section}

\begin{document}

\title{ Sharp non-uniqueness of weak solutions to 2D 
magnetohydrodynamic equations}

\author{Changxing Miao}
\address[Changxing Miao]{Institute  of Applied Physics and Computational Mathematics, Beijing, China.}

\email{miao\_changxing@iapcm.ac.cn}

\author{Yao Nie}

\address[Yao Nie]{School of Mathematical Sciences and LPMC, Nankai University, Tianjin, China.}

 \email{nieyao@nankai.edu.cn}

\author{Weikui Ye}

\address[Weikui Ye]{School of Mathematical Sciences, Shenzhen University, Shenzhen,  China}

 \email{904817751@qq.com}


\date{\today}
\maketitle

\begin{abstract}In this paper, we prove that weak solutions to the 2D viscous and resistive magnetohydrodynamic (MHD) equations are non-unique in 
$L^2_t L^p(\R^2) \cap L^1_t W^{1,p}(\R^2)$ for given any  $1\le p<\infty$, 
showing the sharpness of the Ladyzhenskaya--Prodi--Serrin condition at the endpoint 
$(2,\infty)$ and the solutions live on the borderline of the 
Beale--Kato--Majda criterion.  
As byproducts, we also obtain non-uniqueness for the Navier--Stokes equations in $L^2_t L^p$ 
with $1\le p<\infty$, and for the MHD system with large $\mathrm{BMO}^{-1}$ 
initial data.
\end{abstract}



\section{Introduction}
In this paper, we consider the incompressible MHD equations  on
$\R^2$
:
\begin{equation}
\left\{
\begin{aligned}
&
\partial_t u-\Delta u+(u\cdot\nabla) u+\nabla p=B\cdot\nabla B,\\
& \partial_t B-\Delta B+u\cdot\nabla B  =B\cdot\nabla u,\\
&\nabla\cdot
 u=\nabla \cdot B  = 0,
\end{aligned}
\right. \label{MHD}
\end{equation}
supplemented with a divergence-free initial datum
$(\uin,\bin) $. Here $u$ is the velocity field, $B$ the magnetic field and $p$ the pressure.

For the Navier--Stokes equations (i.e., $B=0$ in \eqref{MHD}), Leray's pioneering work \cite{L} demonstrated the existence of global weak solutions for any initial data $\uin\in L^2(\R^3)$. These solutions obey the energy inequality
\[
\|u(t)\|_{L^2}^{2}+2\int_{0}^{t}\|\nabla u(s)\|_{L^2}^{2}\,\mathrm{d}s \leq \|\uin\|_{L^2}^{2},\qquad \forall\, t\ge 0.
\]
Later, Hopf \cite{Hopf} extended this existence theory to smooth bounded domains in any dimension $d\ge 2$ under Dirichlet boundary conditions. Such weak solutions are now commonly referred to as {Leray--Hopf weak solutions}.  For dimension $d\ge 3$, however, the question of uniqueness and regularity for Leray-Hopf weak solutions remains open. 

The natural scaling of the Navier--Stokes equations in $\mathbb{R}^d$ is given by
\[
u(t,x) \mapsto \lambda u(\lambda^2 t, \lambda x),\qquad p(t,x) \mapsto \lambda^2 p(\lambda^2 t, \lambda x),\quad \lambda>0.
\]
Under this scaling, local well-posedness results for the Cauchy problem have been established in various scaling-invariant spaces (see, e.g., \cite{Cannone95, Can97, CM95, FK, Kato84, KochTataru01, Planchon96}). The critical scaling in the mixed Lebesgue spaces $L^q([0,T]; L^p(\R^d))$ leads to the Ladyzhenskaya--Prodi--Serrin (LPS for short) condition
\begin{align}\label{LPS}
\frac{2}{q}+\frac{d}{p}\le 1,\quad d\le  p\le\infty.
\end{align}
If a Leray–Hopf weak solution $u$ belongs to $L^q_tL^p_x$ with \eqref{LPS}, then the solution is unique for all $p\ge d$ and smooth for $p>d$ \cite{KS, Lady, Prodi, Serr}. For dimension $d=3$, smoothness also holds at the endpoint $(q,p)=(\infty, 3)$, as proved by Escauriaza–Seregin–\v{S}ver\'{a}k \cite{ESS}. In fact, for very weak solutions, the LPS condition \eqref{LPS} excluding the endpoint $p=d=3$ also provides a uniqueness criterion \cite{FJR, LM}, while  the endpoint case $p=d=3$ requires the stronger space $C([0,T];L^3(\mathbb{R}^3))$ \cite{FLT}. 

On the non-uniqueness of weak solutions to the Navier--Stokes equations, one elegant strategy developed by Jia and \v{S}ver\'{a}k \cite{JS14,JS} exploits the nonlinear instability of a steady  solution in the self-similar dynamics. In the forced case, Albritton, Bru\'{e}, and Colombo \cite{ABC} adapted Vishik's construction \cite{V18a,V18b} to prove non-uniqueness of Leray--Hopf weak solutions, consistent with the Jia--\v{S}ver\'{a}k framework. Recently, Hou, Wang, and Yang \cite{HWY} claim to prove non-uniqueness of Leray--Hopf solutions for the unforced 3D Navier--Stokes equations via a computer-assisted verification of key spectral properties. Another major direction stems from convex integration scheme, which originates in Nash’s $C^1$ isometric embedding theorem \cite{Nash}. It was first introduced into the Euler equations in the seminal works of De Lellis and Székelyhidi \cite{DS09, DS13}, and subsequently extended to study the non-uniqueness of solutions for the Navier–Stokes equations. For example, on the periodic setting $\mathbb{T}^d$, Buckmaster and Vicol \cite{BV} constructed the first non-unique weak solutions to the Navier--Stokes equations in $C_t L^2$ for $d=3$. Cheskidov--Luo proved non-uniqueness in $L^\infty_t L^p$ ($1\le p<2$) for $d=2$ \cite{CL23}, and in $L^q_t L^\infty$ ($1\le q<2$) for $d\ge 2$ \cite{1Cheskidov}, thereby demonstrating the sharpness of the LPS criteria in the endpoint cases $(q,p)=(\infty,2)$ and $(2,\infty)$, respectively. In contrast, on the whole space $\mathbb{R}^3$, the authors in \cite{MNY-ar2, MNY-ar1} develop a new iterative framework to  prove non-uniqueness in the classes $C_t L^2$ and $L^q_t L^\infty$ for $1\le q<2$.
In a separate development, Coiculescu and Palasek \cite{CP} recently pioneered a novel mechanism  to construct two distinct global solutions in $\BMO^{-1}(\TTT^3)$.

For the MHD equations \eqref{MHD}, the existence of weak solutions in the Leray--Hopf class were established ~\cite{DL, Wu}, while the questions of uniqueness and regularity remain open.  Analogous to the Navier--Stokes equations, one may expect that if a Leray--Hopf calss of weak solution $(u,B)$ to \eqref{MHD} in $L^q_tL^p_x$ with LPS condition \eqref{LPS} is unique. For the weak solutions, the LPS condition (excluding the endpoint $p=d=3$) also serves as a uniqueness criterion. At the endpoint $q=d=3$, the general weak solution $(u,B)$ is unique in the smaller space $C_t L^3$. The definition of the  weak solutions is as follows:
\begin{definition}[Weak solution]\label{def-weak}
 A pair \((u(x,t), B(x,t))\) is called a weak solution of \eqref{MHD} with initial data $(\uin, \bin)\in \mathcal{D}'(\R^d)$ if the following conditions hold:
\begin{itemize}
    \item [(i)] $u, B \in L^2([0,T];L^2(\R^d))$.
     \item [(ii)] For all test functions \(\varphi, \psi \in \mathcal{D}_T\) (i.e., \(\varphi,\psi\in C_c^\infty([0,T)\times\mathbb{R}^d)\) with \(\operatorname{div}\varphi = \operatorname{div}\psi = 0\)),
\[
\begin{aligned}
&\int_0^T\int_{\mathbb{R}^d} \Bigl( u\cdot(\partial_t\varphi + \Delta \varphi)+ \nabla \varphi: (u\otimes u-B\otimes B)\Bigl) \,\dd x\,\dd t  = -\int_{\mathbb{R}^d} \uin(x)\cdot\varphi(x,0)\,\dd x,\\
&\int_0^T\int_{\mathbb{R}^d} \Bigl( B\cdot(\partial_t\psi + \Delta \psi)+ \nabla \psi: (B\otimes u-u\otimes B)\Bigl) \,\dd x\,\dd t  = -\int_{\mathbb{R}^d} \bin(x)\cdot\psi(x,0)\,\dd x.
\end{aligned}
\]
\item[(iii)] For almost every \(t\in[0,T]\), \(\operatorname{div}(u(\cdot,t)) = 0\) and \(\operatorname{div}(B(\cdot,t)) = 0\) in the sense of distributions.  
\end{itemize}   
\end{definition}

Since the LPS condition \eqref{LPS} serves as a uniqueness criterion for weak solutions of the MHD equations~\eqref{MHD}, a natural question arises:
\vskip 1.5mm
\quad(Q): \textit{Is the LPS criterion \eqref{LPS} sharp for the uniqueness of weak solutions  $(u,B)$ of \eqref{MHD} with a non‑trivial magnetic field? Alternatively, are such weak solutions $(u,B)$ of \eqref{MHD} non-unique in the class $L^q_t L^p_x$ with $\frac{2}{q} + \frac{d}{p} > 1$?}
\vskip 1.5mm

The non-uniqueness problem for the MHD equations has also drawn considerable attention in the literature. For the ideal MHD system, smooth solutions conserve three important quantities: the total energy $\mathcal{E}$, the cross helicity $\mathcal{H}_c$, and the magnetic helicity $\mathcal{H}_m$. This intrinsic feature naturally directs attention to the critical regularity thresholds for the conservation of these invariants. The first construction of nontrivial weak solutions violating total energy conservation was given by Bronzi, Lopes Filho, and Nussenzveig Lopes \cite{BLN}. Subsequently, Faraco, Lindberg, and Sz\'ekelyhidi \cite{FLS} constructed compactly supported $L^\infty$ weak solutions in $\mathbb{R}^3$ for which $\mathcal{E}$ and $\mathcal{H}_c$ (the Els\"asser energies) are not conserved, while $\mathcal{H}_m$ vanishes identically. The first non‑conservative weak solutions with non‑trivial $\mathcal{H}_m$, obtained by Beekie, Buckmaster, and Vicol~\cite{BBV}, lie in $C_t L^p(\mathbb{T}^3)$ for some $p>2$ and fail to conserve $\mathcal{E}$, $\mathcal{H}_c$, and $\mathcal{H}_m$. Later, Faraco, Lindberg, and Sz\'ekelyhidi \cite{FLS24} refined their earlier construction and identified $L^3_{t,x}$ as the sharp threshold for magnetic helicity conservation. The present authors in \cite{MNY2} constructed $C^{1/3-}_{t,x}$ weak solutions on $\mathbb{T}^3$ and $C^{1/5-}_{t,x}$ weak solutions on $\mathbb{T}^2$, both of which do not conserve the Els\"asser energies, thereby resolving the flexible part of the Onsager-type conjecture for Els\"asser energies in the three-dimensional case. More recently,  Enciso, Pe\~nafiel-Tom\'as, and Peralta-Salas \cite{EPP} established the first H\"older continuous solutions that maintain a non-zero $\mathcal{H}_m$ without conserving the Els\"asser energies. 

{For the viscous and resistive MHD equations, the existence of various weak solutions is well understood, yet uniqueness is far from settled. In particular, for Leray--Hopf weak solutions with $d \ge 3$, and for the weak solutions defined in \eqref{def-weak} with $d \ge 2$, the uniqueness problem remains open. Only a few works have been devoted to the study of non‑uniqueness of weak solutions for the MHD system.} Li, Zeng, and Zhang \cite{LZZ} proved the non-uniqueness of weak solutions in $H^\varepsilon_{t,x}$ for some small $\varepsilon$ by constructing new velocity and magnetic flows with refined spatial-temporal intermittency that respect the MHD geometry and control strong viscosity and resistivity. Subsequently, Miao and Ye \cite{MY} introduced box flows and constructed perturbations consisting of seven different types of flows to overcome dissipation, thereby proving the non-uniqueness of weak solutions in {$C_tL^2(\mathbb{T}^3)$ }  with controllable total energy and cross helicity.  The authors of \cite{LZZ24, NY} established the non-uniqueness of weak solutions in $L^q L^\infty(\mathbb{T}^3)$ for $1\le q<2$, which demonstrates the sharpness of the endpoint $(q,p)=(2,\infty)$ in the LPS criterion \eqref{LPS} for $\mathbb{T}^d$ with $d\ge 3$. {However, the convex integration methods  in \cite{LZZ24, NY} break down in the two-dimensional setting.} Recently, Dai \cite{Dai} constructed solutions whose $L^\infty$ norm blows up instantaneously, and showed non-uniqueness of spatial smooth solutions.

{For weak solutions of the MHD equations \eqref{MHD}, the Serrin criterion is a sufficient condition for uniqueness when $d\ge 2$. However, no results are available on the sharpness of this criterion in the two-dimensional case on the whole space. In this paper, we construct two distinct weak solutions sharing the same initial data to demonstrate the sharpness of the classical LPS uniqueness criterion for the MHD equations on $\mathbb{R}^2$ at the endpoint $(2,\infty)$.} Our result is as follows. 

\begin{theorem}[Sharp non-uniqueness]\label{t:main}
For given  $1\le p<\infty$, there exist two distinct  weak solutions of the equations \eqref{MHD} with the same initial data and
\begin{align*}
(u,B),~(\widetilde{u},\widetilde{B})  \in   L^{2}([0,T];L^p(\mathbb{R}^2))\cap L^{1}([0, T];W^{1,p}(\mathbb{R}^2)).
\end{align*}
\end{theorem}
\begin{remark}
We highlight several key aspects of Theorem \ref{t:main} regarding comparison with previous works and its applications.

\begin{enumerate}
\item [(i)] Theorem \ref{t:main} remains valid on the periodic torus $\TTT^2$ as well, and our method extends without essential change to all dimensions $d\ge 2$. Therefore,  Theorem \ref{t:main}  implies the sharpness of the
LPS criteria \eqref{LPS} at the endpoint $(q,p)=(2,\infty)$ for $d\ge 2$. Particularly, for given $d<p<\infty$, there exist two distinct  weak solutions of the equations \eqref{MHD} with the same initial data and
\begin{align*}
(u,B),~(\widetilde{u},\widetilde{B})  \in   L^{2}([0,\infty);L^p(\mathbb{R}^d))\cap L^{1}([0,\infty);W^{1,p}(\mathbb{R}^d))\cap C^\infty(\R^{+}\times\R^d).
\end{align*}

\item [(ii)] For the MHD equations in dimensions $d\ge 3$, convex integration method have been employed to establish non‑uniqueness of weak solutions in the endpoint class $L^q([0,T];L^\infty)$ with $1\le q<2$ \cite{LZZ24, NY}.  The method developed in this paper is independent of  convex integration scheme. In fact, a close inspection of our construction reveals that there exist two distinct weak solutions belonging to $L^q([0,T];L^\infty)$ with the same initial data for $1\le q<2$ and any $d\ge2$. Hence our approach offers a unified framework that not only recovers the known non‑uniqueness phenomena in $L^qL^\infty$, but also extends them to the previously unexplored regime $q=2$ and $p<\infty$.

\item [(iii)] The global well-posedness of \eqref{MHD} for small initial data in $\mathrm{BMO}^{-1}$ was established in \cite{Miao2007}. In fact, the weak solutions obtained in Theorem~\ref{t:main} enjoy the further regularity $L^\infty(\mathbb{R}^{+};\mathrm{BMO}^{-1}(\mathbb{R}^2))$ by following from the method of \cite{CP}. Our result thus reveals non-uniqueness for large data $\mathrm{BMO}^{-1}$, thereby demonstrating the sharpness of the well-posedness result for small data in $\mathrm{BMO}^{-1}$ for the system \eqref{MHD}. Moreover, the argument in this paper can also be spplied to prove non-uniqueness for the Navier–Stokes equations with large $\mathrm{BMO}^{-1}$ data, thereby extending the recent remarkable work~\cite{CP} from the torus to the whole space.

\item [(iv)] In our construction, $u\neq \widetilde{u}$. As a byproduct, our method also shows the non-uniqueness for the Navier–Stokes equations in the class $L^2([0,T];L^p(\R^d))$ with $1\le p<\infty$, as well as in $L^q([0,T];L^\infty(\R^d))$ with $1\le q<2$. These results extend the important result of Cheskidov and Luo~\cite{1Cheskidov}.
\end{enumerate}
\end{remark}
Below we compare different non-uniqueness results of the MHD equations \eqref{MHD} using the scales of space-time Lebesgue
spaces $L^q_tL^p_x$.
\begin{table}[ht]
   \footnotesize
\begin{tabular}{p{3cm}|p{3.2cm}|p{2.8cm}|p{2.8cm}}
\toprule[1.5pt]
\multicolumn{4}{c} {Non-uniqueness results of the MHD equations \eqref{MHD}} \\\midrule[1pt]
Results&Range&dimension &Scaling \\\midrule[1pt]
\cite{LZZ}&$\frac{1}{q}+\frac{2}{p}>\frac{3}{2}$&{$d=3$}& $\frac{2}{q}+\frac{d}{p}>\frac{9}{4}$\\[4pt]\hline
\cite{MY}&{$q=\infty, p=2$}& $d=3$&$\frac{2}{q}+\frac{d}{p}=\frac{3}{2}$\\[4pt]\hline
\cite{LZZ24, NY}&{$1\le q<2,p=\infty$}&${d=3}$& $\frac{2}{q}+\frac{d}{p}>1$\\[4pt]\hline
Theorem \ref{t:main}&{$q=2, 1\le p<\infty$}&${d=2}$& $\frac{2}{q}+\frac{d}{p}>1$\\[4pt]\bottomrule
\end{tabular}
\end{table}

As shown in the table, all previous non‑uniqueness results for the MHD equations \eqref{MHD} were obtained in the three‑dimensional setting via convex integration methods. This is primarily because the higher dimension offers more choices of oscillation directions for constructing building blocks, thereby rendering the convex integration scheme more feasible in three dimensions than in two. For the two‑dimensional ideal MHD equations,  as shown in~\cite{FL2018,FLS}, the $\Lambda$-convex hull $K^{\Lambda}$ has an empty interior, which consequently leaves ``little room'' for a convex integration argument.

To our knowledge, the only convex integration result in two dimensions for the MHD equations to date is the construction by \cite{MNY2} of Hölder weak solutions that exhibit energy dissipation in the ideal (i.e., inviscid and non‑resistive) case. The construction in \cite{MNY2} exploits the finite propagation speed property. Nevertheless, it is not applicable to the viscous and resistive MHD system. We develop a new mechanism that is entirely different from those used in previous non‑uniqueness results for the MHD equations. In what follows, we describe our main ideas in more details. 
\vskip 2mm
\noindent {\textbf{Main ideas}.} The construction in Theorem \ref{t:main} is based on a new iterative scheme that yields exact solutions to the MHD equations \eqref{MHD} at each iteration step. This scheme differs fundamentally from the approaches used in previous works on non‑uniqueness for the MHD equations \cite{LZZ,LZZ24, MY,NY}, where one constructs approximate solutions to a relaxed system rather than exact solutions to the original equations. More precisely,  we give a solution $(u_1, B_1)$ of system \eqref{MHD}, and assume that for $n \ge 3$ we already have $\{u_j, B_j\}_{1 \le j < n}$ as solutions to the MHD equations \eqref{MHD}. Our main task is to construct $(u_n, B_n)$ by building an increment $(w_n, d_n)$ on  $(u_{n-2}, B_{n-2})$, so that the resulting pair $(u_n, B_n)$ remains a solution to the MHD equations. The construction of $(w_n, d_n)$ consists of two parts: the first part is an explicit construction based on the previously known increment $w_{n-1}$, and the second part is obtained via the solvability of a forced  MHD equations.

The explicit construction consists of two components: \textit{the heat-dominated Fourier mode flow $\wpp_n$} and \textit{the inverse cascade-dominated flow $\wss_n$}. Since our target space is $L^2_t L^p$ with $p<\infty$, the intermittent building blocks employed in \cite{LZZ24, NY} become ineffective when $p>2$. Hence, the construction presented here is completely different from those in previous works. The main part $\wpp_n$ is constructed as follows:
\[
\wpp_n \sim \lambda_n e^{-\lambda_n^2 t} \sum_{k\in\Lambda} a_k \left( \mathrm{Id} + \frac{\mathcal{R} w_{n-1}}{1000\|\mathcal{R} w_{n-1}\|_{L^\infty}} \right) \eta_{n-1} \chi_{n-1} \, \phi(\lambda_n^r k\cdot x)\, e^{\ii\lambda_n k\cdot x} \,\bar{k} \;+\; \text{l.o.t.},
\]
where $\phi(\lambda_n^r k\cdot x)\bar{k}$ is a two-dimensional Mikado flow and the exponent $r$ is dependent on $p$. Inspired by~ \cite{CP}, we introduce the Mikado flow together with the smooth cutoff function $\chi$ so that the support of $\wpp_n$ shrinks monotonically as the iteration proceeds, which enables us to achieve a limiting solution belonging to $L^2_t L^p$ ($p<\infty$). The 2D Mikado flows inevitably intersect geometrically, thereby producing a delicate error from the quadratic form $\Div(\wpp_n\otimes \wpp_n)$. To overcome this, we introduce a highly oscillatory {flow} $e^{-\lambda_n^2 t}\sum_{k\in\Lambda} e^{\ii\lambda_n k\cdot x}\bar{k}$ into $\wpp_n$. This {flow} is designed to serve two crucial purposes simultaneously:
\vspace{-0.5ex}
\begin{itemize}
    \item {Absorption into pressure} -- the highly oscillatory {flow}  is an exact solution of the {steady two-dimensional Euler equations}, so that the oscillatory part of $\Div(\wpp_n\otimes\wpp_n)$ can be absorbed into the pressure term while keeping the residual error under control.
    \item Cancellation of the principal dissipative error -- the highly oscillatory {flow} is also a solution of the {free heat equation}, thereby eliminating the main error caused by the dissipative term, in contrast to previous convex integration iterations where $-\Delta u$ is handled  as an error term.
\end{itemize}
{The inverse cascade-dominated flow $\wss_n$} is given by $\wss_{n}=\wpp_{n-1}e^{-2\lambda^2_{n}t}$, as introduced by \cite{CP}. The construction of $\wss_{n}$ guarantees that the $\lambda_n$-level oscillatory error generated by the Mikado flows in the quadratic term $\Div(\wpp_n\otimes \wpp_n)$ is cancelled by $\partial_t \wss_n$, leaving only lower-order parts. 

The implicit part of the increment $(w_n, d_n)$ is the perturbation flow $(\wmhd_n, \dmhd_n)$. Our key insight is to fully exploit the exponential decay of the time factor $e^{-\lambda_n^2 t}$ in the construction of $\wpp_n$ and $\wss_n$. This allows us to expect that the $L^1_t C^1$ norm of $(u_{n-2},B_{n-2})$ grows only linearly in $n$ (i.e., $O(n)$). At the same time, the residual error decays doubly exponentially (e.g., $\lambda_{n-1}^{-30}$) in the $L^1_t B^{-1+\frac{2}{p}}_{p,1}$ space. The latter decays far faster than the former grows, which allows us to treat the residual as a forcing term. Consequently, by solving a forced MHD system with a suitably chosen small initial datum (see \eqref{e:wt}), we obtain a small corrector $(\wmhd_n, \dmhd_n)$. Finally, let $(w_n, d_n) = (\wpp_n + \wss_n + \wmhd_n, \dmhd_n)$ and we can construct $(u_n, B_n)$ to be an exact solution of the MHD equations.

\noindent\textbf{Organization of the paper.}  Section \ref{Pre} collects the necessary preliminaries and technical tools that will be used throughout the paper. In Section \ref{Induction}, we present an iterative framework and state the key iteration proposition (Proposition \ref{iteration}); we then show that Theorem \ref{t:main} follows directly from this proposition. Finally, the proof of Proposition \ref{iteration} is given in Section \ref{proof-1}.

\noindent\textbf{Notation.} In this paper, $\TTT^2=\R^2/(2\pi\ZZ)^2$. For a $\TTT^2$-periodic function $f$, we denote
\begin{align*}
\mathbb{P}_{=0} f:=\frac{1}{|\TTT^2|}\int_{\TTT^2} f(x)\dd x,\quad\PP f:=f-\mathbb{P}_{=0} f.
\end{align*}
In the following, the notation $a\lesssim b$ means $a\le Cb$ for a universal constant $C$ that may change from line to line.  {Without ambiguity, we will denote $L^r([0,T];Y(\TTT^2))$ and $L^r([0,\infty);Y(\TTT^2))$ by $L^r_T Y$ and $L^r_t Y$,  respectively.} For a vector funtion $f=(f_1, f_2)$, $\curl f=\partial_1 f_2-\partial_2 f_1$.

Let $\partial^{\sigma}$ be the space derivatives, we denote
\begin{align*}
\|f\|_{L^r_tC^N}:=\sum_{j=0}^N\max_{|\sigma|=j}\|\partial^\sigma f\|_{L^r_{t}L^\infty}.
\end{align*}
and 
\begin{align*}
\|(f_1,f_2,\cdots,f_m)\|_{X}:=\max\Big\{\|f_1\|_X, \|f_2\|_X, \cdots, \|f_m\|_X\Big\}.
\end{align*}

\section{Preliminaries}\label{Pre}
In this section we collect several useful tools: the definitions of mollifiers, a geometric lemma, Besov spaces and Chemin–Lerner spaces, as well as some estimates for transport‑diffusion equations and oscillating exponential factors.

\begin{definition}[Mollifiers]\label{e:defn-mollifier-t}
Let $\varphi(t)\in C^\infty_c(-1,0)$ and $\psi(x)\in C^\infty_c(B_1(0))$ be nonnegative standard mollifying kernels such that
\[
\int_{\R}\varphi(t)\,\dd t = \int_{\R^2}\psi(x)\,\dd x = 1.
\]
For each $\epsilon>0$ we define the families
\[
\varphi_{\epsilon}(t) \coloneq \frac1{\epsilon}\,\varphi\!\left(\frac t\epsilon\right),\qquad 
\psi_{\epsilon}(x) \coloneq \frac1{\epsilon^2}\,\psi\!\left(\frac{x}{\epsilon}\right).
\]
\end{definition}

\begin{lemma}[Stationary flows in 2D \cite{CDS}]\label{Betrimi}
Let $\mathbb{S}^1$ be the unit circle and $\mathbb{Q}$ the set of rational numbers. Suppose $\Lambda\subseteq\mathbb{S}^1\cap\mathbb{Q}^2$ satisfies $-\Lambda=\Lambda$. Then for any real numbers $b_k$ with $b_k=b_{-k}$, the vector field
\[
W(\xi)=\sum_{k\in\Lambda}b_k\frac{\mathrm{i}\,k^{\perp}}{|k|}\,e^{\mathrm{i}k\cdot\xi},\qquad k\perp k^{\perp},
\]
is real‑valued, divergence‑free and satisfies
\[
\Div_{\xi}(W\otimes W)= \frac{1}{2}\nabla_{\xi}\Bigl(|W|^2-\bigl|\sum_{k\in\Lambda}\tfrac{b_k}{|k|}e^{\mathrm{i}k\cdot\xi}\bigr|^2\Bigr)
\]
and
\begin{align}\label{b-g}
W\otimes W =& \sum_{\substack{j,k\in\Lambda\\ j+k\neq0}} -b_j b_k e^{\mathrm{i}(j+k)\cdot\xi}\frac{j^{\perp}}{|j|}\otimes \frac{k^{\perp}}{|k|}
   +\sum_{\substack{j,k\in\Lambda\\ j+k=0}} b_k^2\frac{k^{\perp}}{|k|}\otimes \frac{k^{\perp}}{|k|} \notag\\
=& \sum_{\substack{j,k\in\Lambda\\ j+k\neq0}} -\frac{b_k}{|k|}\frac{b_j}{|j|}\, e^{\mathrm{i}(j+k)\cdot\xi}\,j^{\perp}\otimes k^{\perp}
   +\sum_{\substack{j,k\in\Lambda\\ j+k=0}} \Bigl(\frac{b_k}{|k|}\Bigr)^2 k^{\perp}\otimes k^{\perp}.
\end{align}
\end{lemma}

\begin{lemma}[Geometric lemma \cite{CDS}]\label{first S}
Let $B_{\sigma}(0)$ denote the ball of radius $\sigma$ centered at $\mathrm{Id}$ in the space of $2\times2$ symmetric matrices, and take $\sigma=10^{-3}$.
There exist a set $\Lambda\subseteq \mathbb{S}^1\cap\mathbb{Q}^2$, consisting of vectors $k$ with associated orthonormal basis $(k,k^{\perp})$, and smooth functions $a_{k}:B_{\sigma}(\mathrm{Id})\rightarrow\mathbb{R}$ such that for every $R\in B_{\sigma}(\mathrm{Id})$,
\[
R = \sum_{k\in\Lambda} a_{k}^2(R)\; k^{\perp}\otimes k^{\perp}.
\]
\end{lemma}

We now recall the definitions of Besov spaces and the mixed time‑space Besov spaces (the so‑called Chemin–Lerner spaces).

\begin{definition}[Besov spaces \cite{BCD11, MZZ}]\label{def:besov}
Let $s\in\mathbb{R}$ and $1\le p,q\le\infty$.
The homogeneous Besov space $\dot B^s_{p,q}(\R^2)$ consists of all $u\in\mathcal{S}'_h(\R^2)$ such that
\[
\|u\|_{\dot B^s_{p,q}(\R^2)} \overset{\text{def}}{=}
\Bigl\|\bigl(2^{js}\|\dot\Delta_j u\|_{L^p(\R^2)}\bigr)_{j\in\ZZ}\Bigr\|_{\ell^q(\ZZ)}<\infty
\]
and  the nonhomogeneous Besov space $B^s_{p,q}(\R^2)$ consists of all $u\in\mathcal{S}'(\R^2)$ such that
\[
\|u\|_{B^s_{p,q}(\R^2)} \overset{\text{def}}{=}
\Bigl\|\bigl(2^{js}\|\Delta_j u\|_{L^p(\R^2)}\bigr)_{j\in\ZZ}\Bigr\|_{\ell^q(\ZZ)}<\infty.
\]
\end{definition}

\begin{definition}[Chemin–Lerner spaces \cite{BCD11, MZZ} ]\label{def:cheminlerner}
Let $T>0$, $s\in\R$ and $1\le r,p,q\le\infty$. 
The homogeneous mixed time‑space Besov space $\widetilde{L}^r_T\dot B^s_{p,q}(\R^2)$ consists of all $u\in\mathcal{S}'_h(\R^2)$ such that
\[
\|u\|_{\widetilde{L}^r_T\dot B^s_{p,q}(\R^2)} \overset{\text{def}}{=}
\Bigl\|\bigl(2^{js}\|\dot\Delta_j u\|_{L^r([0,T];L^p(\R^2))}\bigr)_{j\in\ZZ}\Bigr\|_{\ell^q(\ZZ)}<\infty.
\]
The non‑homogeneous version $\widetilde{L}^r_T B^s_{p,q}(\R^2)$ consists of all $u\in\mathcal{S}'(\R^2)$ satisfying
\[
\|u\|_{\widetilde{L}^r_T B^s_{p,q}(\R^2)} \overset{\text{def}}{=}
\Bigl\|\bigl(2^{js}\|\Delta_j u\|_{L^r([0,T];L^p(\R^2))}\bigr)_{j\in\ZZ}\Bigr\|_{\ell^q(\ZZ)}<\infty.
\]
\end{definition}

\begin{lemma}[\cite{BCD11}]\label{T-D}
Let $1\le p\le p_1\le\infty$ and
\[
-1-2\min\Bigl\{\frac1{p_1},\frac1{p'}\Bigr\} < s < 1+\frac{2}{p_1}.
\]
Consider the transport‑diffusion equation
\begin{align}\label{T-D-E}
\partial_t u - \Delta u + v\cdot\nabla u = g,\qquad u(0,x)=u_0(x),
\end{align}
where $v$ is a divergence‑free vector field. There exists a constant $C$, depending only on $s$ and $p$, such that for any smooth solution $u$ of \eqref{T-D-E}, any $1\le \rho_1,q\le\infty$ and $\rho\in[\rho_1,\infty]$,
\[
\|u\|_{\widetilde{L}_T^{\rho}\dot B^{s+\frac{2}{\rho}}_{p,q}(\R^2)}
\le C\,e^{C V_{p_1}(T)}\Bigl( \|u_0\|_{\dot B_{p,q}^s(\R^2)}
+ \|g\|_{\widetilde{L}_T^{\rho_1}(\dot B_{p,q}^{s-2+\frac{2}{\rho_1}}(\R^2))}\Bigr),
\]
where
\[
V_{p_1}(T)=\int_0^T \bigl\|\nabla v(s)\bigr\|_{\dot B^{\frac{2}{p_1}}_{p_1,\infty}\cap L^\infty}\,\dd s.
\]
\end{lemma}

\begin{proposition}\label{Est-fe}
For $k\in\mathbb{Q}^2\setminus\{0\}$ and $\lambda>0$,
\[
\|f e^{\mathrm{i}\lambda k\cdot x}\|_{L^\infty_t \dot B^{-1}_{\infty,\infty}}
\lesssim_{|k|}\; \lambda^{-1}\|f\|_{L^\infty_tL^\infty}
+ \lambda^{-2}\|f\|_{L^\infty_tC^1}
+ \lambda^{-3}\|f\|_{L^\infty_tC^2}.
\]
\end{proposition}
\begin{proof}
Since $e^{\mathrm{i}\lambda k\cdot x} = \frac{\Div(e^{\mathrm{i}\lambda k\cdot x}k)}{\mathrm{i}\lambda |k|}$, we have
\begin{align*}
f e^{\mathrm{i}\lambda k\cdot x}
&= \frac{1}{\mathrm{i}\lambda |k|}\Div\bigl(f e^{\mathrm{i}\lambda k\cdot x}k\bigr)
   - \frac{1}{\mathrm{i}\lambda |k|} e^{\mathrm{i}\lambda k\cdot x}\,k\cdot\nabla f \\
&= \frac{1}{\mathrm{i}\lambda |k|}\Div\bigl(f e^{\mathrm{i}\lambda k\cdot x}k\bigr)
   - \frac{1}{(\mathrm{i}\lambda |k|)^2}\Bigl(
        \Div\bigl((k\cdot\nabla f)e^{\mathrm{i}\lambda k\cdot x}k\bigr)
        - e^{\mathrm{i}\lambda k\cdot x}\,k\cdot\nabla(k\cdot\nabla f)\Bigr).
\end{align*}
Consequently, 
we conclude that\begin{align*}
\|f e^{\mathrm{i}\lambda k\cdot x}\|_{L^\infty_t\dot B^{-1}_{\infty,\infty}}
\lesssim&\; \lambda^{-1}\|f e^{\mathrm{i}\lambda k\cdot x}\|_{\dot B^0_{\infty,\infty}}
   +\lambda^{-2}\bigl(\|(\nabla f)e^{\mathrm{i}\lambda k\cdot x}\|_{\dot B^0_{\infty,\infty}}
   +\|(\nabla^2 f)e^{\mathrm{i}\lambda k\cdot x}\|_{\dot B^{-1}_{\infty,\infty}}\bigr)\\
\lesssim&\; \lambda^{-1}\|f\|_{L^\infty_{t,x}}
   +\lambda^{-2}\|f\|_{L^\infty_tC^1}
   +\lambda^{-3}\|f\|_{L^\infty_tC^2}
\end{align*}
and complete the proof of Proposition \ref{Est-fe}.
\end{proof}
\section{Induction scheme}\label{Induction}
Without loss of generality, we first aim to establish Theorem~\ref{t:main} for a fixed exponent \(2<p<\infty\). To this end, we construct a new iterative scheme.

\subsection{Parameters}\label{para}
 First of all, we introduce several parameters.  Given $f\in C^\infty_0(\mathbb{R}^2)$, $2<p<\infty$ and $T>0$, we first choose  \(r\in\QQ\) such that  
\[
0 < r < \frac{1}{10}\Bigl(1-\frac{2}{p}\Bigr).
\]
Then there exists a positive integer \(M\) with \(rM\in\mathbb{N}\).
Since the set $\Lambda\subseteq \mathbb{S}^1\cap\mathbb{Q}^2$ in Lemma~\ref{first S} is finite, we can find a positive integer $N_{\Lambda}$ such that  
\[
N_{\Lambda} k \in \mathbb{Z}^2,\qquad \forall k\in\Lambda.
\]
Now we introduce two large positive integers $a,b\in M\mathbb{N}$  that will be chosen later depending on $T$, $C_0$, $p$, and $\|f\|_{\dot B^{0}_{2,1}\cap \dot B^{1}_{2,1}\cap B^{-1}_{1,1}}$ and satisfy 
\[
b > 2^{10}\max\Bigl\{\bigl(1-\tfrac{2}{p}\bigr)^{-1},\; \tfrac{p}{2}\Bigr\},
\qquad
a > \|f\|_{\dot B^{0}_{2,1}\cap \dot B^{1}_{2,1}\cap B^{-1}_{1,1}}.
\]
For any $q\ge 0$,  we define
\begin{align}\label{def-ell}
\lambda_q := N_{\Lambda}^{M} a^{b^q}, \qquad 
\ell_q := \min\Bigl\{\lambda_q^{-\frac{50}{1-\frac{2}{p}}},\; \lambda_q^{-\frac{50}{\frac{2}{p}}}\Bigr\}.    
\end{align}

\subsection{Iterative procedure}\label{sec-ite}
We now construct the iteration scheme by induction. For each integer $j\ge 0$, set
\[
\mathcal{O}_j = \Bigl[-\tfrac{3}{4}\pi + 5\lambda_j^{-\frac{1}{50}},\; \tfrac{3}{4}\pi - 5\lambda_j^{-\frac{1}{50}}\Bigr]^2,
\]
and let $\eta_j \in C_c^\infty(\mathbb{R}^2)$ be a smooth cut-off function such that
\begin{align}\label{def-eta}
\spt\,\eta_j\subseteq\mathcal{O}_{j+1} ,
\qquad
\eta_j(x) \equiv 1 \;\; \text{for all } x \in \mathcal{O}_{j}.
\end{align}
Define $\varepsilon_0 = \lambda_1^{-\frac{1}{15}}$ and denote the orthogonal gradient by $\nabla^{\perp} = (\partial_2, -\partial_1)$.  
We introduce $(\wpp_1, \wss_1)$ as
\[
\wpp_1 = \varepsilon_0 \lambda_1^{-3} e^{-\lambda_1^2 t} \, \Delta \nabla^{\perp} \Bigl( \eta_0 \operatorname{curl}\bigl( \sin(\lambda_1 x_1) e_2 \bigr) \Bigr), \qquad
\wss_1 = 0.
\]

Let $(\wmhd_1, p^{(\mathrm{m})}_1, \dmhd_1)$ be the unique global solution of the forced MHD system
\begin{equation}\label{def-w1}
\left\{
\begin{aligned}
&\partial_t \wmhd_1 - \Delta \wmhd_1 + \wmhd_1\!\cdot\!\nabla \wmhd_1 
   + \wpp_1\!\cdot\!\nabla \wmhd_1 + \wmhd_1\!\cdot\!\nabla \wpp_1 + \nabla p^{(\mathrm{m})}_1 \\
&\hspace{2cm}= \dmhd_1\!\cdot\!\nabla \dmhd_1 
   - (\partial_t - \Delta)\wpp_1 - \Div(\wpp_1 \otimes \wpp_1), \\[4pt]
&\partial_t \dmhd_1 - \Delta \dmhd_1 + (\wmhd_1 + \wpp_1)\!\cdot\!\nabla \dmhd_1
   = \dmhd_1\!\cdot\!\nabla (\wmhd_1 + \wpp_1), \\[4pt]
&\nabla\!\cdot\!\wmhd_1 = \nabla\!\cdot\!\dmhd_1 = 0, \\[4pt]
&(\wmhd_1, \dmhd_1)\big|_{t=0} = \bigl(0,\; \lambda_1^{-50} f(x) \bigr).
\end{aligned}
\right.
\end{equation}
Owing to the specific choice of $\wpp_1$ and the smallness of $\varepsilon_0$,  the system \eqref{def-w1} admits a unique global solution $(\wmhd_1, p^{(\mathrm{m})}_1, \dmhd_1)$ satisfying the estimates
\begin{align}
\|(\wmhd_1, \dmhd_1)\|_{\widetilde L^{\infty}_t (\dot{B}^{-1+\frac{2}{p}}_{p,1} \cap \dot B^{-\frac{2}{p}}_{2,1})\cap L^{1}_t (\dot{B}^{1+\frac{2}{p}}_{p,1} \cap \dot B^{2-\frac{2}{p}}_{2,1})} 
   \lesssim{\varepsilon_0 }(\lambda^{-1+\frac{2}{p}}_1+\lambda^{-\frac{2}{p}}_1), \label{wmhd1}
\end{align}
and
\begin{align}
 \|(\wmhd_1, \dmhd_1)\|_{L^{\infty}_T B^{-1}_{1,1}  \cap L^{1}_T {B}^{1}_{1,1}} 
   \lesssim\varepsilon_0,\qquad 
 \|(\wmhd_1, \dmhd_1)\|_{\widetilde L^{\infty}_t \dot{B}^{0}_{\infty,1} \cap L^{1}_t \dot{B}^{2}_{\infty,1}} 
   \leq \lambda_1^{-\frac{1}{20}}.   \label{wmhd1-1}
\end{align}
Finally we define $(u_1, p_1, B_1)$ by
\[
(u_1, p_1, B_1) = \bigl( \wpp_1 + \wmhd_1,\; p^{(\mathrm{m})}_1,\; \dmhd_1 \bigr)
\]  
 that satisfies the MHD equations \eqref{MHD}.

Now assume that for some integer $n\ge 2$ and  each $1\le j\le n-1$, we have already constructed smooth solutions
$\{(u_j,B_j)\}_{1\le j\le n-1}\subseteq C^\infty_{t,x}(\mathbb{R}^{+}\times \mathbb{R}^2)$ of \eqref{MHD} on $\mathbb{R}^{+}\times\mathbb{R}^2$ satisfying
\begin{align}
& \|(u_j, B_j)\|_{L^2_TL^2}\le 2C^{3/4}_0\sum_{k=1}^j 2^{-\frac{j}{2}},\quad \|(u_j, B_j)\|_{L^\infty_t\dot B^0_{\infty,1}\cap L^1_t\dot B^2_{\infty,1}}\le C_0\lambda_j,\label{u-Linfty} \\[4pt]
& \|(u_j,B_j)\|_{L^1_t\dot B^{1}_{\infty,1}\cap L^2_t \dot B^0_{\infty,1}} \le
    \begin{cases}
        C_0 q,      & \text{if } j = 2q, \\[2pt]
        C_0 (q+1), & \text{if } j = 2q+1,
    \end{cases} \label{u-L1C1} \\[8pt]
& u_j =
    \begin{cases}
        \displaystyle  \sum_{l=1}^{q} w_{2l}
        =  \sum_{l=1}^{q} \bigl( \wpp_{2l} + \wss_{2l} + \wmhd_{2l} \bigr),
        & \text{if } j = 2q, \\[10pt]
        \displaystyle  \sum_{l=0}^{q} w_{2l+1}
        =  \sum_{l=0}^{q} \bigl( \wpp_{2l+1} + \wss_{2l+1} + \wmhd_{2l+1} \bigr),
        & \text{if } j = 2q+1,
    \end{cases} \label{u-decomposition}
    \end{align}
and
    \begin{align}
& B_j =
    \begin{cases}
        \displaystyle  \sum_{l=1}^{q} \dmhd_{2l},
        & \text{if } j = 2q, \\[10pt]
        \displaystyle   \sum_{l=0}^{q} \dmhd_{2l+1} ,
        & \text{if } j = 2q+1.
    \end{cases} \label{b-decomposition}
\end{align}
Set $\mathcal R := (\nabla+\nabla^{\!\top})(-\Delta)^{-1}$. Then $\wpp_j$ and $\wss_j$ can be rewritten as
\begin{align}\label{def-R}
 \wpp_j = \Div(\mathcal R \wpp_j),\qquad \wss_j = \Div(\mathcal R \wss_j)
\end{align}
and satisfy
\begin{align}
&\|\partial_t^{\,l}\, \mathcal R \wpp_j\|_{L^\infty _tC^{N}}
    \le C^{3/4}_0 \lambda_j^{N+2l}, \qquad 
\|\mathcal R\wss_j(t)\|_{L^\infty_tC^{N}} \le C^{3/4}_0 \lambda_{j-1}^{N},\qquad
0\le N\le 3,\; l=0,1,\label{Rwj}\\
& \|\wpp_j\|_{L^{2}_tL^p \cap L^{1}_t W^{1,p}} \le C_0^{3/4}\, 2^{-\frac{j}{p}}, \qquad
\|\wss_j\|_{L^{2}_tL^q \cap L^{1}_t W^{1,q}} \le C_0^{3/4} 2^{-\frac{j}{p}},\qquad
\forall\, 1\le q\le \infty, \label{wp-L}
\end{align}
and for each $t\ge 0$,
\begin{align}
&\|w^{(\rm h)}_{j}(t)\|_{ \dot B^{N}_{\infty,1}} \leq C^{3/4}_0 e^{-\lambda^2_{j}t}\lambda^{N+1}_{j}, \quad
\|\wss_j(t)\|_{  \dot B^{N}_{\infty,1}} \leq C^{3/4}_0 e^{-(\lambda^2_j+\lambda^2_{j-1}) t} \lambda^{N+1}_{j-1},\quad -2\le N\le 2. \label{wp-wsL1}    
\end{align}
Moreover, the perturbation flows $(\wmhd_j,\dmhd_j)$ obey
\begin{align}
  &\|(\wmhd_j, \dmhd_j)\|_{\widetilde L^{\infty}_t\dot{B}^{ -1+\frac{2}{p}}_{p,1}\cap L^{1}_t\dot{B}^{1+\frac{2}{p}}_{p,1}}
\le \lambda^{-20}_{j-1},\quad  \|(\wmhd_j, \dmhd_j)\|_{\widetilde L^{\infty}_t\dot{B}^{ 0}_{\infty,1}\cap L^{1}_t\dot{B}^{2}_{\infty,1}}
\le \lambda^{2r}_{j}.
\label{e-wmdm}
\end{align}

\begin{proposition}[Inductive step]\label{iteration}
Let $2<p<\infty$. There exists an absolute constant $C_0$ with the following property. Suppose that $n \ge 2$ and that $\{(u_j, B_j)\}_{0 \le j \le n-1} \subseteq C^\infty_{t,x}(\mathbb{R}^+ \times \mathbb{R}^2)$ are solutions of the MHD equations \eqref{MHD} satisfying estimates \eqref{u-Linfty}–\eqref{e-wmdm}. Then one can construct  
\[
w_n = w^{(\rm h)}_{n} + w^{(\rm s)}_{n} + w^{{\rm (m)}}_n \,\,\text{and}\,\, d^{{\rm (m)}}_n
\]
such that 
\begin{align}\label{indentity}
 u_n := u_{n-2} + w_n,\quad B_n := B_{n-2} + d^{{\rm (m)}}_n,
\end{align}
and the pair $(u_n,B_n)$ again solves \eqref{MHD} and fulfills estimates \eqref{u-Linfty}–\eqref{e-wmdm} with $n-1$ replaced by $n$. Moreover,  we have
\begin{align}
&\|( w^{{\rm (m)}}_n,  d^{{\rm (m)}}_n)\|_{\widetilde L^{\infty}_t\dot{B}^{-\frac{2}{p}}_{2,1}\cap L^{1}_t\dot{B}^{2-\frac{2}{p}}_{2,1}} \le \lambda^{-20}_{n-1},\quad
\|(w^{{\rm (m)}}_{2q},d^{{\rm (m)}}_{2q})\|_{L^{\infty}_T{B}^{-1}_{1,1}\cap L^{1}_T{B}^{1}_{1,1}}\le \lambda^{-15}_{2q-1}.\label{wm-C}\\
& w^{\rm{(h)}}_{n-1}(0,x) = w^{\rm{(i)}}_n(0,x), \quad\text{and}\quad
\big(w^{{\rm (m)}}_n(0,x),\,d^{{\rm (m)}}_n(0,x)\big)=\big(0,\lambda^{-50}_{2\lceil\tfrac{n}{2}\rceil-1}f(x)\big). \label{ini-wm}
\end{align}
\end{proposition}

We now show that Theorem~\ref{t:main} follows directly from Proposition~\ref{iteration}.

\noindent \textbf{Proposition \ref{iteration} implies Theorem \ref{t:main}.}  
Applying Proposition~\ref{iteration} repeatedly generates two families of smooth solutions 
$\{(u_{2q-1},B_{2q-1})\}_{q\ge 1}$ and $\{(u_{2q},B_{2q})\}_{q\ge 1}$ 
to the MHD system \eqref{MHD}, given explicitly by
\begin{align*}
(u_{2q-1}, B_{2q-1}) = \Bigl( \sum_{l=1}^{q} w_{2l-1},\; \sum_{l=1}^{q} d^{\rm (m)}_{2l-1} \Bigr), \,\,\text{and}\,\,
(u_{2q}, B_{2q})     = \Bigl( \sum_{l=1}^{q} w_{2l},\; \sum_{l=1}^{q} d^{\rm (m)}_{2l} \Bigr).
\end{align*}

Since $p>2$, we have the continuous embeddings
\begin{align*}
\widetilde{L}^{\infty}_t\dot{B}^{-\frac{2}{p}}_{2,1}(\R^2)\cap L^{1}_t\dot{B}^{2-\frac{2}{p}}_{2,1}(\R^2)
&\hookrightarrow \widetilde{L}^{\infty}_t\dot{B}^{-1}_{p,1}(\R^2)\cap L^{1}_t\dot{B}^{1}_{p,1}(\R^2) 
\hookrightarrow L^{2}_t L^p(\R^2) \cap L^{1}_t W^{1,p}(\R^2),
\end{align*}
and
\[
L^\infty_T B^{-1}_{1,1}(\R^2)\cap L^1_T B^1_{1,1}(\R^2)
\hookrightarrow L^2_T B^0_{1,1}(\R^2)\cap L^1_T B^1_{1,1}(\R^2)
\hookrightarrow L^2_T L^1(\R^2)\cap L^1_T W^{1,1}(\R^2).
\]
Combining these embeddings with the estimates \eqref{wp-L} and \eqref{wm-C} yields
\begin{align*}
\sum_{q=1}^\infty \bigl\| (w_{2q}, w_{2q-1}, d^{(\rm m)}_{2q}, d^{(\rm m)}_{2q-1}) \bigr\|_{L^2_t L^p \cap L^1_t W^{1,p} \cap L^2_T L^1 \cap L^1_T W^{1,1}}
\le \sum_{q=1}^\infty \bigl( 2C_0^{3/4} 2^{-\frac{2q}{p}} + \lambda_{2q-1}^{-10} \bigr) < \infty.
\end{align*}
In view of \eqref{wmhd1}, \eqref{wmhd1-1} and \eqref{indentity}, the above estimate implies that both 
$\{(u_{2q-1},B_{2q-1})\}_{q\ge 1}$ and $\{(u_{2q},B_{2q})\}_{q\ge 1}$ are Cauchy sequences in the space
\[
L^{2}_t L^p(\R^2) \cap L^{1}_t W^{1,p}(\R^2) \cap L^2_T L^1(\R^2) \cap L^1_T W^{1,1}(\R^2).
\]
Hence there exist the two limit functions
\[
(u, B),\;(\widetilde{u}, \widetilde{B}) \in L^2\bigl([0,T]; L^1(\R^2)\cap L^p(\R^2)\bigr) \cap L^1\bigl([0,T]; W^{1,1}(\R^2)\cap W^{1,p}(\R^2)\bigr)
\]
such that, as $q\to\infty$,
\[
(u_{2q-1},B_{2q-1}) \to (u, B), \qquad
(u_{2q}, B_{2q}) \to (\widetilde{u}, \widetilde{B}),
\]
strongly in $L^{2}_T(L^1(\R^2)\cap L^p(\R^2)) \cap L^{1}_T(W^{1,1}(\R^2)\cap W^{1,p}(\R^2))$.  
Moreover, both $(u, B)$ and $(\widetilde{u},  \widetilde{B})$ satisfy the  MHD system \eqref{MHD}.

From the initial conditions \eqref{ini-wm}, we deduce
\[
u_{2q}(0,x) = u_{2q+1}(0,x) - \wpp_{2q+1}(0,x), \qquad
B_{2q}(0,x) = B_{2q+1}(0,x) = \sum_{j=1}^q \lambda_{2j-1}^{-50} f(x).
\]
Thanks to \eqref{wp-wsL1}, the term $\wpp_{2q+1}(0,x)\to 0$  in $\dot{B}^{-2}_{\infty,1}(\R^2)$ as $q\to\infty$. Consequently,
\[
\bigl(u(0,x), B(0,x)\bigr) = \bigl(\widetilde{u}(0,x), \widetilde{B}(0,x)\bigr)
\qquad \text{in the sense of distributions}.
\]

It remains to prove that $(u, B) \neq (\widetilde{u}, \widetilde{B})$.  
Because both solutions belong to $C^\infty(\mathbb{R}^{+}\times \mathbb{R}^2)$, it suffices to show that 
\[
\|(u, B)(\lambda_1^{-2}, \cdot)\|_{\dot{B}^{-1}_{\infty,1}}
\neq \|(\widetilde{u}, \widetilde{B})(\lambda_1^{-2}, \cdot)\|_{\dot{B}^{-1}_{\infty,1}}.
\]

Recall that
\[
u_1 = \wpp_{1} + \wmhd_1
    = \varepsilon_0 \lambda_1^{-3} e^{-\lambda_1^2 t} \Delta \nabla^{\perp}
      \bigl( \eta_0 \operatorname{curl}(\sin(\lambda_1 x_1) e_2) \bigr) + \wmhd_1.
\]
At $t = \lambda_1^{-2}$, we obtain the lower bound
\begin{align*}
\bigl\|\wpp_{1}(\lambda_1^{-2},\cdot)\bigr\|_{\dot{B}^{-1}_{\infty,1}}
&\ge C\varepsilon_0 e^{-1} \lambda_1^{-3}
   \bigl\| \Delta \bigl( \eta_0 \operatorname{curl}(\sin(\lambda_1 x_1) e_2) \bigr) \bigr\|_{\dot{B}^{0}_{\infty,1}} \\
&\ge C\varepsilon_0 e^{-1} \lambda_1^{-3}
   \bigl\| \eta_0 \Delta \operatorname{curl}(\sin(\lambda_1 x_1) e_2) \bigr\|_{\dot{B}^{0}_{\infty,1}} \\
&\quad - C\varepsilon_0 e^{-1} \lambda_1^{-3}
   \Bigl( \bigl\| (\Delta\eta_0) \operatorname{curl}(\sin(\lambda_1 x_1) e_2) \bigr\|_{\dot{B}^{0}_{\infty,1}} \\
&\qquad\quad + 2\sum_{l=1}^2 \bigl\| \partial_l\eta_0 \, \partial_l \operatorname{curl}(\sin(\lambda_1 x_1) e_2) \bigr\|_{\dot{B}^{0}_{\infty,1}} \Bigr) \\
&\ge \frac{C}{2} \varepsilon_0 e^{-1} =: M_0.
\end{align*}
Using \eqref{wmhd1}, we obtain
\begin{align*}
\bigl\| u_{1}(\lambda_1^{-2},\cdot) \bigr\|_{\dot{B}^{-1}_{\infty,1}}
&\ge \bigl\| \wpp_{1}(\lambda_1^{-2},\cdot) \bigr\|_{\dot{B}^{-1}_{\infty,1}}
   - \bigl\| \wmhd_1 \bigr\|_{\dot{B}^{-1}_{\infty,1}} \ge \frac{C}{2} \varepsilon_0 e^{-1}
   - C\varepsilon_0 \bigl( \lambda_1^{-1+\frac{2}{p}} + \lambda_1^{-\frac{2}{p}} \bigr)
   \ge \frac{1}{2} M_0,
\end{align*}
provided that the integer $a$ is chosen sufficiently large.  
Thanks to \eqref{wp-wsL1} and \eqref{e-wmdm}, we have
\begin{align*}
\sum_{q\ge 1} \bigl\| (\wpp_{2q}, \wss_{2q}, \wmhd_{2q}, \wpp_{2q+1}, \wss_{2q+1}, \wmhd_{2q+1})(\lambda_1^{-2},\cdot) \bigr\|_{\dot{B}^{-1}_{\infty,1}}
\lesssim \sum_{q\ge 1} \bigl( C_0^{2} e^{-\lambda_{2q}^2 \lambda_1^{-2}} + \lambda_{2q-1}^{-10} \bigr)
\lesssim \lambda_1^{-8},
\end{align*}
where the last inequality holds for large enough $a$. Consequently,
\begin{align*}
&\bigl\| (u- \widetilde{u})(\lambda_1^{-2},\cdot) \bigr\|_{\dot{B}^{-1}_{\infty,1}}\\
\ge& \bigl\| \wpp_{1}(\lambda_1^{-2},\cdot) \bigr\|_{\dot{B}^{-1}_{\infty,1}}
   - \sum_{q\ge 1} 6 \bigl\| (\wpp_{2q}, \wmhd_{2q}, \wss_{2q}, \wpp_{2q+1}, \wss_{2q+1}, \wmhd_{2q+1})(\lambda_1^{-2},\cdot) \bigr\|_{\dot{B}^{-1}_{\infty,1}} \\
\ge& \frac12 M_0 - C\lambda_1^{-8}
   \ge \frac{M_0}{4},
\end{align*}
as soon as $a$ is sufficiently large. This completes the proof of Theorem~\ref{t:main}.  
We now proceed to establish Proposition~\ref{iteration}.

\section{Proof of Proposition \ref{iteration}}\label{proof-1}
Without loss of generality, we may assume that the index $n$ is even; write $n = 2q$ with $q\in\mathbb{N}$.  
We shall construct the increment $(w_{2q}, \dmhd_{2q})$ to be added to the existing solution $(u_{2q-2}, B_{2q-2})$, thereby obtaining $(u_{2q}, B_{2q}) = (u_{2q-2} + w_{2q},\, B_{2q-2} + \dmhd_{2q})$.  
The velocity increment $w_{2q}$  is decomposed into three components:
\begin{itemize}
    \item the heat‑dominated Fourier mode flow $\wpp_{2q}$,
    \item the inverse‑cascade dominated flow $\wss_{2q}$,
    \item the perturbation flow $\wmhd_{2q}$.
\end{itemize}
Thus $w_{2q} = \wpp_{2q} + \wss_{2q} + \wmhd_{2q}$, while the magnetic increment is simply $\dmhd_{2q}$.

\subsection{Construction of the flows $\wpp_{2q}$ and $\wss_{2q}$}\label{Con-wpws} To ensure that $\{\wpp_{2q}\}$ and $\{\wss_{2q}\}$ are Cauchy sequences in the integrable spaces $L^p$, we introduce the cutoff function $\chi_{2q-1}$. This idea, originating from \cite{CP}, will be employed to obtain a monotonically decreasing sequence of compact support volumes in the iteration scheme. First, let $\phi$ be a smooth, even function on the torus $\mathbb{T}$ such that $\int_{\mathbb{T}} \phi^2 \,\dd x = 1$ and $\operatorname{spt}\phi \subseteq \bigcup_{m \in \mathbb{Z}} \bigl([-\tfrac{1}{100}, \tfrac{1}{100}] + 2\pi m\bigr)$. The building blocks are then defined as
\begin{align}\label{def-phi}
\phi_{k,l}(x) = \phi\bigl( \lambda^{r}_{l} \, k \cdot x \bigr), \qquad l \in \mathbb{N}^{+}, \; k \in \Lambda,
\end{align}
where $\Lambda$ is the finite set defined in Lemma \ref{first S}. For $m\ge 1$, we set
\begin{align}\label{def-Omega}
&\Omega_{m} =  \bigcap_{1 \le l \le m} \, \bigcup_{k \in \Lambda} \operatorname{spt}\, \phi_{k,l}=:\bigcap_{1 \le l \le m} \, \bigcup_{k \in \Lambda} \Omega_{k,l},\notag\\
&\widetilde{\Omega}_{m}=\bigcap_{1 \le l \le m}\bigcup_{k \in \Lambda} \widetilde{\Omega}_{k,l},\qquad
\widetilde{\Omega}_{k,l} := \bigl\{ x : d(x, \Omega_{k,l}) < \tfrac{1}{100} \lambda_l^{-r} \bigr\}.
\end{align}
Then a sequence $\{\chi_m\}$ is defined recursively by: $\chi_1 = 1$ on $\Omega_1$ with $\spt\chi_1\subseteq\widetilde{\Omega_1}$. For $m\ge 2$,
\begin{align}\label{def-chi}
 \chi_m= 1 \text{ on } \spt\,(\chi_{m-1}\phi_m) \quad\text{and}\quad \spt\,\chi_m\subseteq\widetilde{\Omega}_m,   
\end{align}
and these functions satisfy $\|\chi_m\|_{C^N} \lesssim \lambda_m^{Nr}$ for $m\ge 1, N\ge 0$.

We now define \textit{the heat‑dominated Fourier mode flow} $\wpp_{2q}$ by
\begin{equation}\label{def-wh}
\wpp_{2q} 
= -\lambda^{-3}_{2q} e^{-\lambda^2_{2q }t} \, \Delta \, \nabla^{\perp} \Bigl( \sum_{k \in \Lambda} a_{(k,2q)} \, \eta_{2q-1}\chi_{2q-1} \, \phi_{k,2q} \; \operatorname{curl} \bigl( e^{\mathrm{i}\lambda_{2q } k \cdot x} \kp \bigr) \Bigr),
\end{equation}

where
\begin{itemize}
\item[(i)] $\eta_{2q-1}$, $\phi_{k,2q}$ and $\chi_{2q-1}$ are  defined in \eqref{def-eta}, \eqref{def-phi} and \eqref{def-chi}, respectively;
\item[(ii)] for each $k\in\Lambda$, the amplitude $a_{(k,2q)}$ is given by
\begin{align}\label{def-akq}
a_{(k,2q)}(t,x) 
&= (2000C_0)^{1/2} \int_{t}^{t+\ell_{2q-1}} \Bigl( a_k \bigl( \mathrm{Id} + \tfrac{ \mathcal{R}w^{(p)}_{2q-1} }{ 1000C_0} \bigr) \ast \psi_{\ell_{2q-1}} \Bigr)(x,s) \; \varphi_{\ell_{2q-1}}(t-s) \, \dd s 
\end{align}
with $a_k$ the coefficients from Lemma~\ref{first S} and $\psi_{\ell}$, $\varphi_{\ell}$ the mollifiers from Definition~\ref{e:defn-mollifier-t};

\item[(iii)] from the definitions of $a_{(k,2q)}$ and $\phi_{k,2q}$ we immediately obtain, for any $N\ge 0$,
\begin{align}\label{est-aphi}
 \|a_{(k,2q)}\|_{L^\infty_t C^N_x} \lesssim (2000C_0)^{1/2} \, \ell_{2q-1}^{-N}, \qquad
\|\phi_{k,2q}\|_{C^N} \lesssim \lambda^{rN}_{2q};   
\end{align}

\item[(iv)] since $k \perp \bar{k}$ and $|k|=1$, one has
$-\Delta \nabla^{\perp} \operatorname{curl} \bigl( e^{\mathrm{i}\lambda_{2q } k \cdot x} \kp \bigr) = \lambda^4_{2q} e^{\mathrm{i}\lambda_{2q } k \cdot x} \kp$, which yields the decomposition 
\begin{align}\label{dec-wh}
    \wpp_{2q}= \wpm_{2q} + \wpr_{2q}
\end{align}
with
\begin{align}
\wpm_{2q} &= \lambda_{2q} e^{-\lambda^2_{2q}t} \sum_{k \in \Lambda} a_{(k,2q)} \chi_{2q-1}\eta_{2q-1} \phi_{k,2q} e^{\mathrm{i}\lambda_{2q } k \cdot x} \kp, \label{def-wpm}\\[4pt]
\wpr_{2q} &= -\lambda^{-1}_{2q} e^{-\lambda^2_{2q}t} \sum_{k \in \Lambda} \Delta\bigl( a_{(k,2q)} \chi_{2q-1}\eta_{2q-1} \phi_{k,2q} \bigr) e^{\mathrm{i}\lambda_{2q } k \cdot x} \kp \notag\\
&\quad - 2\mathrm{i} e^{-\lambda^2_{2q}t} \sum_{k \in \Lambda} \nabla\bigl( a_{(k,2q)} \chi_{2q-1}\eta_{2q-1} \phi_{k,2q} \bigr) \cdot k \, e^{\mathrm{i}\lambda_{2q } k \cdot x} \kp\notag \\
&\quad - \lambda^{-3}_{2q} e^{-\lambda^2_{2q}t} \Delta \Bigl( \sum_{k \in \Lambda} \nabla^{\perp}\bigl( a_{(k,2q)} \chi_{2q-1}\eta_{2q-1} \phi_{k,2q} \bigr) \, \operatorname{curl}\bigl( e^{\mathrm{i}\lambda_{2q } k \cdot x} \kp \bigr) \Bigr).\label{def-wpr}
\end{align}
\item [(v)] Thanks to $\Delta f=\Div(\nabla+\nabla^{\TT})f$ for $\Div f=0$, we infer from \eqref{def-wh} that
\begin{align}\label{def-Rw}
\wpp_{2q} 
&= -\lambda^{-3}_{2q} e^{-\lambda^2_{2q }t} \, \Div(\nabla+\nabla^{\TT}) \, \nabla^{\perp} \Bigl( \sum_{k \in \Lambda} a_{(k,2q)} \, \eta_{2q-1}\chi_{2q-1} \, \phi_{k,2q} \; \operatorname{curl} \bigl( e^{\mathrm{i}\lambda_{2q } k \cdot x} \kp \bigr) \Bigr)
\notag\\
&=:\Div(\mathcal {R}\wpp_{2q}),
\end{align}
which shows that {$\eta^2_{2q}\chi^2_{2q}\mathcal{R}\wpp_{2q}=\mathcal{R}\wpp_{2q}$.}
\end{itemize}

We now define \textit{the inverse‑cascade dominated flow} $\wss_{2q}$ by  
\begin{equation}\label{wttq}
\wss_{2q} = \frac{1}{2} \chi^2_{2q-1}\eta^2_{2q-1} \sum_{k \in \Lambda} \Div \Bigl( a^2_k \bigl( \mathrm{Id} + \tfrac{ \mathcal{R}\wpp_{2q-1} }{ 1000C_0} \bigr) (2000C_0) \, \kp \otimes \kp \Bigr) e^{-2\lambda^2_{2q }t}.
\end{equation}
Applying Lemma~\ref{first S} and the property $\chi^2_{2q-1}\eta^2_{2q-1}\wpp_{2q-1} = \wpp_{2q-1}$, we obtain 
\begin{align}
\wss_{2q} 
&= \frac{1}{2} \chi^2_{2q-1}\eta^2_{2q-1} \Div\bigl( \mathrm{Id} + \tfrac{\mathcal{R}\wpp_{2q-1}}{1000C_0} \bigr) (2000C_0) e^{-2\lambda^2_{2q }t} \notag \\
&= \chi^2_{2q-1}\eta^2_{2q-1} \, \wpp_{2q-1} \, e^{-2\lambda^2_{2q }t}
 = \wpp_{2q-1} \, e^{-2\lambda^2_{2q }t}. \label{def-ws}
\end{align}
From this identity, we immediately deduce  
\begin{align}\label{wss2}
\wpp_{2q-1}(0,x) = \wss_{2q}(0,x), \qquad
\partial_t \wss_{2q} = -2\lambda^2_{2q } e^{-2\lambda^2_{2q }t} \wpp_{2q-1} + \partial_t \wpp_{2q-1} \, e^{-2\lambda^2_{2q }t}.
\end{align}

\subsection{Main estimates for $\wpp_{2q}$ and $\wss_{2q}$}\label{est-wpws} In this section we provide preliminary estimates for $\wpp_{2q}$ and $\wss_{2q}$.

\begin{proposition}[Estimates for $\wpp_{2q}$ and $\wss_{2q}$]\label{Est-wpws}
There exists a universal constant $C_0$ such that the following bounds hold:
\begin{align}
&\|\partial^l_t \mathcal{R}w^{(\rm h)}_{2q}\|_{L^{\infty}_{t}C^N} \le C^{3/4}_0 \lambda^{N+2l}_{2q}, \quad\,\,
\|\mathcal{R}w^{(\rm i)}_{2q}\|_{L^{\infty}_{t}C^N} \le C^{3/4}_0 \lambda^N_{2q-1},
\quad 0\le N\le 3,\; l=0,1; \label{wpq-wss1} \\[4pt]
&\|w^{(\rm h)}_{2q}(t)\|_{\dot B^{N}_{\infty,1}} \le C^{3/4}_0 e^{-\lambda^2_{2q}t} \lambda^{N+1}_{2q}, \,\,
\|w^{(\rm i)}_{2q}(t)\|_{\dot B^{N}_{\infty,1}} \le C^{3/4}_0 e^{-(\lambda^2_{2q}+\lambda^2_{2q-1})t} \lambda^{N+1}_{2q-1},
\quad -2\le N\le 2. \label{es-wp-BN}
\end{align}
Moreover, for the decomposition $\wpp_{2q} = \wpm_{2q} + \wpr_{2q}$ given in \eqref{dec-wh}, we have, for every $t \ge 0$ and $N \ge 0$,
\begin{align}\label{E-wpm-wpr}
\|\wpm_{2q}(t)\|_{C^N} &\le \frac{1}{2} C^{3/4}_0 e^{-\lambda^2_{2q}t} \lambda^{N+1}_{2q}, &
\|\wpr_{2q}(t)\|_{C^N} &\le C^{3/4}_0 e^{-\lambda^2_{2q}t} \lambda^{N+r}_{2q}.
\end{align}
\end{proposition}

\begin{proof}
From the definition \eqref{def-Rw}, we obtain, for $0\le N\le 3$,
\begin{align}\label{est-Rw}
\|\partial^l_t\mathcal{R}\wpp_{2q}\|_{L^\infty_{t}C^N}
\le& \sum_{k \in \Lambda} \lambda^{-3}_{2q} \Bigl\|  \partial^l_ta_{(k,2q)} \chi_{2q-1} \eta_{2q-1} \phi_{k,2q} \operatorname{curl} \bigl( \mathrm{i} e^{\mathrm{i}\lambda_{2q} k \cdot x} \kp\bigr) \Bigr\|_{L^\infty_{t}C^{N+2}} \notag\\
&+\sum_{k \in \Lambda} \lambda^{-3+2l}_{2q} \Bigl\|  a_{(k,2q)} \chi_{2q-1} \eta_{2q-1} \phi_{k,2q} \operatorname{curl} \bigl( \mathrm{i} e^{\mathrm{i}\lambda_{2q} k \cdot x} \kp\bigr) \Bigr\|_{L^\infty_{t}C^{N+2}} \notag\\
\lesssim& \sum_{k \in \Lambda}( \lambda^{-2}_{2q}\|\partial^l_ta_{(k,2q)} \chi_{2q-1} \eta_{2q-1} \phi_{k,2q} \|_{C^{N+2}}+C^{1/2}_0\ell^{-l}_{2q-1}\lambda^N_{2q})\notag\\
&+\sum_{k \in \Lambda}( \lambda^{-2+2l}_{2q}\|a_{(k,2q)} \chi_{2q-1} \eta_{2q-1} \phi_{k,2q} \|_{C^{N+2}}+C^{1/2}_0\lambda^{N+2l}_{2q})\notag\\
\lesssim& C^{1/2}_0\lambda^{(N+2)r-2+2l}_{2q}+C^{1/2}_0\lambda^{N+2l}_{2q}\lesssim C^{1/2}_0\lambda^{N+2l}_{2q}.
\end{align}
For the Besov norm, we deduce from \eqref{def-wh} that, for $-2\le N\le 2$,
\begin{align*}
\|\wpp_{2q}(t)\|_{\dot B^{N}_{\infty,1}}
&\le \lambda^{-3}_{2q} e^{-\lambda^2_{2q}t} \Bigl(
   \|a_{(k,2q)} \chi_{2q-1}\eta_{2q-1}\phi_{k,2q}\|_{L^\infty_{t,x}} \|e^{\mathrm{i}\lambda_{2q}k\cdot x}\|_{\dot B^{N+4}_{\infty,1}} \\
&\qquad + \|a_{(k,2q)} \chi_{2q-1}\eta_{2q-1}\phi_{k,2q}\|_{L^\infty_{t}C^{N+4}} \|e^{\mathrm{i}\lambda_{2q}k\cdot x}\|_{C^1} \Bigr) \\
&\lesssim \lambda^{-3}_{2q} e^{-\lambda^2_{2q}t} \bigl( C^{1/2}_0 \lambda^{N+4}_{2q} + C^{1/2}_0 \lambda^{r(N+4)+1}_{2q} \bigr) \\
&\lesssim C^{1/2}_0 e^{-\lambda^2_{2q}t} \lambda^{N+1}_{2q}.
\end{align*}

The bounds for $\wss_{2q}$ stated in \eqref{wpq-wss1}--\eqref{es-wp-BN} follow from the representation \eqref{def-ws} together with the induction hypotheses \eqref{Rwj} and \eqref{wp-wsL1} on $\wpp_{2q-1}$.

We estimate the main term $\wpm_{2q}$ as
\begin{align*}
\|\wpm_{2q}(t)\|_{C^N}
&\lesssim \lambda_{2q} e^{-\lambda^2_{2q}t} \Bigl(
   \| a_{(k,2q)} \chi_{2q-1}\eta_{2q-1} \phi_{k,2q} \|_{L^\infty_t C^N}
   + \lambda^N_{2q} \| a_{(k,2q)} \chi_{2q-1}\eta_{2q-1} \phi_{k,2q} \|_{L^\infty_{t,x}} \Bigr) \\
&\lesssim C^{1/2}_0 e^{-\lambda^2_{2q}t} \lambda^{N+1}_{2q}.
\end{align*}
and the remainder $\wpr_{2q}$ by
\begin{align*}
\|\wpr_{2q}(t)\|_{C^N}
&\lesssim \lambda^{-1}_{2q} e^{-\lambda^2_{2q}t} \Bigl(
   \| a_{(k,2q)} \chi_{2q-1}\eta_{2q-1} \phi_{k,2q} \|_{L^\infty_t C^{N+2}}
   + \lambda^N_{2q} \| a_{(k,2q)} \chi_{2q-1}\eta_{2q-1} \phi_{k,2q} \|_{L^\infty_{t}C^2} \Bigr) \\
&\quad + e^{-\lambda^2_{2q}t} \Bigl(
   \| a_{(k,2q)} \chi_{2q-1}\eta_{2q-1} \phi_{k,2q} \|_{L^\infty_t C^{N+1}}
   + \lambda^N_{2q} \| a_{(k,2q)} \chi_{2q-1}\eta_{2q-1} \phi_{k,2q} \|_{L^\infty_{t}C^1} \Bigr) \\
&\quad + \lambda^{-3}_{2q} e^{-\lambda^2_{2q}t} \Bigl(
   \lambda_{2q} \| a_{(k,2q)} \chi_{2q-1}\eta_{2q-1} \phi_{k,2q} \|_{L^\infty_t C^{N+3}}
   + \lambda^{N+3}_{2q} \| a_{(k,2q)} \chi_{2q-1}\eta_{2q-1} \phi_{k,2q} \|_{L^\infty_{t}C^1} \Bigr) \\
&\lesssim C^{1/2}_0 e^{-\lambda^2_{2q}t} \lambda^{N+r}_{2q}.
\end{align*}
Hence, by choosing $C_0$ large enough, we complete the proof of Proposition~\ref{Est-wpws} 
\end{proof}

\begin{proposition}\label{cedu}There exists a large enough  constant $C_0$ such that
\begin{align*}
\|(w^{\rm (h)}_{2q},w^{\rm (i)}_{2q})\|_{L^{2}_t L^ m\cap L^{1}_t W^{1,m}} &\le C^{\frac{3}{4}}_0 2^{-\frac{2q}{m}}, \quad\forall 1\le m\le\infty.
\end{align*}
\end{proposition}

\begin{proof}
For $1 \le m < \infty$, we observe that both $\wpp_{2q}$ and $\wss_{2q}$ are supported in the square $[-\pi,\pi]^2$.  
Consequently, we can use the argument  in \cite[Proposition 3.3]{MNY3}  to show the desired bounds.

For the endpoint $m= \infty$, an directly computation yields that
\begin{align*}
\|\wpp_{2q}\|_{L^2_t L^\infty \cap L^1_t W^{1,\infty}}
&\lesssim \sum_{k \in \Lambda} \lambda^{-3}_{2q} \|e^{-\lambda^2_{2q}t}\|_{L^2} \,
   \bigl\| a_{(k,2q)} \chi_{2q-1}\eta_{2q-1} \phi_{k,2q} \, \operatorname{curl} \bigl( e^{\mathrm{i}\lambda_{2q} k \cdot x} \bar{k} \bigr) \bigr\|_{L^\infty C^3} \\
&\quad + \sum_{k \in \Lambda} \lambda^{-3}_{2q} \|e^{-\lambda^2_{2q}t}\|_{L^1} \,
   \bigl\| a_{(k,2q)} \chi_{2q-1}\eta_{2q-1} \phi_{k,2q} \, \operatorname{curl} \bigl( e^{\mathrm{i}\lambda_{2q} k \cdot x} \bar{k} \bigr) \bigr\|_{L^\infty C^4}\\
  & \lesssim C^{1/2}_0
\end{align*}
and
\begin{align*}
\|\wss_{2q}\|_{L^2_t L^\infty \cap L^1_t W^{1,\infty}}
&= \|\wpp_{2q-1} e^{-2\lambda^2_{2q}t}\|_{L^2_t L^\infty \cap L^1_t W^{1,\infty}} \le \|\wpp_{2q-1}\|_{L^2_t L^\infty \cap L^1_t W^{1,\infty}}
 \lesssim C^{1/2}_0.
\end{align*}
Thus Proposition~\ref{cedu} is proved.
\end{proof}

\begin{proposition}\label{prop-E}
Let $w^{\rm (h)}_{2q}$ and $w^{\rm (i)}_{2q}$ be defined in \eqref{def-wh} and \eqref{def-ws}. Then there exist a pressure $P_{2q}$ and an error term $E_{2q}$ such that
\begin{align}
    \Div(w^{\rm (h)}_{2q} \otimes w^{\rm (h)}_{2q}) + \partial_t w^{\rm (i)}_{2q} = E_{2q} + \nabla P_{2q}.
    \label{decom-wpwp}
\end{align}
\end{proposition}
\begin{proof}The proof is essentially the same as that of \cite[Proposition 3.5]{MNY3}. For the convenience of the reader, we give a brief proof here.

Thanks to \eqref{dec-wh}, one has
\begin{align}\label{dec}
    \Div(\wpp_{2q} \otimes \wpp_{2q})
   & = \Div(\wpm_{2q} \otimes \wpm_{2q})
       + \Div\bigl( \wpm_{2q} \otimes \wpr_{2q}
                   + \wpr_{2q} \otimes \wpm_{2q}
                   + \wpr_{2q} \otimes \wpr_{2q} \bigr)
\end{align}
From Lemmas \ref{Betrimi} and \ref{first S}, a direct computation gives
\begin{align}\label{dec-1}
&\Div(\wpm_{2q}\otimes\wpm_{2q})=\nabla P^{(1)}_{2q} + E^{(1)}_{2q}
   + \lambda^2_{2q} e^{-2\lambda^2_{2q}t}
     \Div\Bigl( \sum_{k \in \Lambda}
            \chi^2_{2q-1} a^2_{(k,2q)} \phi^2_{k,2q}
            \kp \otimes \kp \Bigr),
            \end{align}
where
\begin{align}
P^{(1)}_{2q}=\Bigl( \sum_{\substack{k'+k\neq0\notag  \\ k,k'\in\Lambda}}
        \frac{1}{2} \lambda^2_{2q} e^{-2\lambda^2_{2q}t}
        \chi^2_{2q-1}\eta^2_{2q-1} a_{(k,2q)} a_{(k',2q)}
        \phi_{k,2q} \phi_{k',2q}
        \bigl( \kp \cdot {k'}^{\perp}- 1 \bigr)
        e^{\mathrm{i}\lambda_{2q}(k'+k)\cdot x} \Bigr) \notag
        \end{align}
        and
    \begin{align}
&E^{(1)}_{2q}=  \mathbb{P}_{\neq 0} \Bigl( \sum_{\substack{k'+k\neq0 \\ k,k'\in\Lambda}}
        \Bigl[ \frac{1}{2} \nabla\bigl( a_{(k,2q)} a_{(k',2q)}
               \chi^2_{2q-1}\eta^2_{2q-1} \phi_{k,2q} \phi_{k',2q} \bigr)
               \bigl( 1- \kp \cdot {k'}^{\perp}\bigr) \notag \\
&\qquad\qquad\qquad\quad
        -\Div \bigl( a_{(k,2q)} a_{(k',2q)}
          \phi_{k,2q} \phi_{k',2q} {k'}^{\perp} \otimes \kp \bigr) \Bigr]
        \lambda^2_{2q} e^{-2\lambda^2_{2q}t}
        e^{\mathrm{i}\lambda_{2q}(k'+k)\cdot x} \Bigr) .\label{def-E1}
\end{align}
By  Lemma \ref{first S} and $\mathbb{P}_{=0}(\phi^2_{k,2q})=1$), we further have
\begin{align}\label{dec-2}
\lambda^2_{2q} e^{-2\lambda^2_{2q}t}
   \Div\Bigl( \sum_{k \in \Lambda}
          \chi^2_{2q-1}\eta^2_{2q-1} a^2_{(k,2q)} \phi^2_{k,2q}
          \kp \otimes \kp \Bigr) =\nabla P^{(2)}_{2q}+E^{(2)}_{2q}+ 2\lambda^2_{2q} e^{-2\lambda^2_{2q}t} \wpp_{2q-1}
          \end{align}
with
\begin{align*}
    P^{(2)}_{2q}=2000C_0 \lambda^2_{2q} e^{-2\lambda^2_{2q}t} \chi^2_{2q-1}\eta^2_{2q-1}
\end{align*}
and
\begin{align}
E^{(2)}_{2q}=& \lambda^2_{2q} e^{-2\lambda^2_{2q}t}
   \sum_{k \in \Lambda} \Div\Bigl(
   \chi^2_{2q-1}\eta^2_{2q-1} \phi^2_{k,2q}
        \Bigl[ a^2_{(k,2q)}
             - 2000C_0a^2_k \bigl( \mathrm{Id}
               + \tfrac{\mathcal{R}\wpp_{2q-1}}{1000C_0} \bigr) \Bigr]
        \kp \otimes \kp \Bigr) \notag\\
&\quad + \lambda^2_{2q} e^{-2\lambda^2_{2q}t}
   \sum_{k \in \Lambda} \Div\Bigl(
        2000C_0 \chi^2_{2q-1}\eta^2_{2q-1}
        \bigl( \mathbb{P}_{\neq 0} \phi^2_{k,2q} \bigr)
        a^2_k \bigl( \mathrm{Id} + \tfrac{\mathcal{R}\wpp_{2q-1}}{1000C_0} \bigr)
        \kp \otimes \kp \Bigr).\label{def-E2}
\end{align}
The equalities \eqref{dec}--\eqref{dec-2} together  with 
\begin{align*}
2\lambda^2_{2q} e^{-2\lambda^2_{2q}t} \wpp_{2q-1}+\partial_t\wss_{2q}=\partial_t\wpp_{2q-1}e^{-2\lambda^2_{2q}t}
\end{align*}
show \eqref{decom-wpwp} with $P_{2q}=P^{(1)}_{2q}+P^{(2)}_{2q}$ and
\begin{align}\label{def-E}
    E_{2q}=E^{(1)}_{2q}+E^{(2)}_{2q}+E^{(3)}_{2q},
\end{align}
where $E^{(1)}_{2q}$,  $E^{(2)}_{2q}$ are defined in \eqref{def-E1} and \eqref{def-E2}, respectively, and
\begin{align}\label{def-E3}
    E^{(3)}_{2q}=\Div\bigl( \wpm_{2q} \otimes \wpr_{2q}
                   + \wpr_{2q} \otimes \wpm_{2q}
                   + \wpr_{2q} \otimes \wpr_{2q} \bigr)+\partial_t\wpp_{2q-1}e^{-2\lambda^2_{2q}t}.
\end{align}
Hence, we complete the proof of Proposition \ref{prop-E}.
\end{proof}
\subsection{Construction of $\wmhd_{2q}$ and  $\dmhd_{2q}$}\label{Con-wmhd} 
Suppose that \((u_{2q-2}, B_{2q-2})\) is a global solution to the MHD equations \eqref{MHD}. We expect that after adding the increments, the resulting pair remains a solution; specifically, let  
\[
u_{2q} = u_{2q-2} + \wpp_{2q} + \wss_{2q} + \wmhd_{2q}, \qquad 
B_{2q} = B_{2q-2} + \dmhd_{2q}.
\]
Then, in order for \((u_{2q}, B_{2q})\) to satisfy the MHD equations, using Proposition \ref{prop-E}, we impose that \((\wmhd_{2q}, \dmhd_{2q})\) satisfy the following equations.
\begin{equation} \label{e:wt}\tag{FMHD}
\left\{ \begin{alignedat}{-1}
&\del_t \wmhd_{2q}-\Delta \wmhd_{2q}+  (u_{2q-2} +\wpp_{2q}+\wss_{2q}+\wmhd_{2q} )\cdot\nabla \wmhd_{2q} +\wmhd_{2q}\cdot\nabla(u_{2q-2}+\wpp_{2q}+\wss_{2q}) \\
&\qquad\qquad\qquad\quad= -\nabla p^{(\textup{m})}_{2q } +\dmhd_{2q}\cdot\nabla \dmhd_{2q}+B_{2q-2}\cdot\nabla \dmhd_{2q} + \dmhd_{2q}\cdot\nabla B_{2q-2}   -E_{2q }-F_{2q},
\\
&\partial_t \dmhd_{2q}-\Delta \dmhd_{2q}+  (u_{2q-2}+\wpp_{2q}+\wss_{2q}+\wmhd_{2q})\cdot\nabla \dmhd_{2q}
+\wmhd_{2q}\cdot\nabla B_{2q-2} \\
&\qquad\qquad\qquad\qquad\qquad=\dmhd_{2q}\cdot\nabla (u_{2q-2}+\wpp_{2q}+\wss_{2q}+\wmhd_{2q})  +B_{2q-2}\cdot\nabla \wmhd_{2q}-G_{2q} ,\\
 & \Div \wmhd_{2q}=\Div \dmhd_{2q}= 0,
  \\
  & (\wmhd_{2q},\dmhd_{2q}) |_{t=0}=  (0,\lambda^{-50}_{2q-1}f(x) ) ,
\end{alignedat}\right.
\end{equation}
where $E_{2q}$ is given by \eqref{decom-wpwp}, $F_{2q}$ is defined by
\begin{align}
F_{2q}=&(\partial_t \wpp_{2q} -\Delta \wpp_{2q})-\Delta \wss_{2q} \notag \\
&+\Div\Big(\wpp_{2q} \otimes( \wss_{2q}+u_{2q-2}) +  (\wss_{2q}+u_{2q-2})\otimes  \wpp_{2q} \notag \\
&+  \wss_{2q} \otimes  u_{2q-2}+u_{2q-2} \otimes  \wss_{2q}+ \wss_{2q} \otimes  \wss_{2q}\Big).\label{def-F}
\end{align}
and
\begin{align}\label{def-G}
  G_{2q}:=( \wpp_{2q}+\wss_{2q})\cdot\nabla B_{2q-2}-B_{2q-2}\cdot\nabla( \wpp_{2q}+\wss_{2q}).   
 \end{align}
In fact, for such \(\wpp_{2q}, \wss_{2q}\) and the given initial conditions, the forced MHD equations \eqref{e:wt} are globally well-posed. To establish this result, we first provide some estimates for the forcing terms \(E_{2q}\), \(F_{2q}\), and \(G_{2q}\).
\begin{proposition}[Estimates for $E_{2q}$]\label{E2q}Let $E_{2q}$ be as defined in \eqref{decom-wpwp}. For the fixed $2<p<\infty$ given in Proposition~\ref{iteration}, we have
\begin{align}\label{es-E}
 \|E_{2q}\|_{L^1_t (\dot B^{-1+\frac{2}{p}}_{p,1}\cap \dot B^{-\frac{2}{p}}_{2,1}\cap B^{-1}_{1,1})} \lesssim C^3_0\lambda^{-90}_{2q-1},  \quad \|E_{2q}\|_{L^1_t \dot B^{0}_{\infty,1}} \le C^3_0\lambda^{r}_{2q}.
\end{align}
\end{proposition}
\begin{proof}Following \eqref{def-E}, we provide estimates for $E^{(1)}_{2q}, E^{(2)}_{2q}$ and $E^{(3)}_{2q}$, respectively. For $E^{(1)}_{2q}$, by Proposition \ref{Est-fe} and  \eqref{est-aphi}, we have
\begin{align}
  \|E^{(1)}_{2q}\|_{L^1_t \dot{B}^{-1}_{\infty,\infty}}
    &\lesssim \Bigl( \sum_{\substack{k'+k\neq0 \notag\\ k,k'\in\Lambda}}
               \lambda^{-1}_{2q}
               \|\chi^2_{2q-1}\eta^2_{2q-1} a_{(k,2q)} a_{(k',2q)}
                 \phi_{k,2q} \phi_{k',2q}\|_{L^\infty_t C^1} \notag\\
    &\qquad + \sum_{\substack{k'+k\neq0 \\ k,k'\in\Lambda}}
               \lambda^{-2}_{2q}
               \|\chi^2_{2q-1}\eta^2_{2q-1} a_{(k,2q)} a_{(k',2q)}
                 \phi_{k,2q} \phi_{k',2q}\|_{L^\infty_t C^2}\notag\\
                 &\qquad + \sum_{\substack{k'+k\neq0 \\ k,k'\in\Lambda}}
               \lambda^{-3}_{2q}
               \|\chi^2_{2q-1}\eta^2_{2q-1} a_{(k,2q)} a_{(k',2q)}
                 \phi_{k,2q} \phi_{k',2q}\|_{L^\infty_t C^3}\Bigr)\notag\\
    &\lesssim C_0(\lambda^{-1}_{2q}\lambda^{r}_{2q} + \lambda^{-2}_{2q}\lambda^{2r}_{2q}+\lambda^{-3}_{2q}\lambda^{3r}_{2q})\lesssim C_0\lambda^{-1+r}_{2q}\label{E1-B-1}
\end{align}
An direct computation yields that
\begin{align}
    \|E^{(1)}_{2q}\|_{L^1_t \dot{B}^{1}_{\infty,\infty}}
    &\lesssim \sum_{\substack{k'+k\neq0 \\ k,k'\in\Lambda}}
               \lambda_{2q}
               \|\chi^2_{2q-1}\eta^2_{2q-1} a_{(k,2q)} a_{(k',2q)}
                 \phi_{k,2q} \phi_{k',2q}\|_{L^\infty_t C^1} \notag\\
    &\quad + \sum_{\substack{k'+k\neq0 \\ k,k'\in\Lambda}}
               \|\chi^2_{2q-1}\eta^2_{2q-1} a_{(k,2q)} a_{(k',2q)}
                 \phi_{k,2q} \phi_{k',2q}\|_{L^\infty_t C^2} \notag\\
    &\lesssim \sum_{\substack{k'+k\neq0 \\ k,k'\in\Lambda}}(C_0 \lambda^{1+r}_{2q}+C_0\lambda^{2r}_{2q})\lesssim C_0\lambda^{1+r}_{2q}.\label{E1-B1}
\end{align}
By the definition of $a_{(k,q)}$ in \eqref{def-akq} and \eqref{Rwj}, one has 
{\begin{align*}
&\|a^2_{(k,q)}-2000C_0a^2_k\big({\rm Id}+\tfrac{ \mathcal{R}\wpp_{2q-1}
 }{  1000C_0}\big)\|_{L^\infty_t L^\infty}
 \lesssim C_0\|a_k\big({\rm Id}+\tfrac{ \mathcal{R}\wpp_{2q-1}
 }{  1000C_0}\big)\|_{C^1_{t,x}}\\
 \lesssim& C_0(\ell^2_{2q-1}\|\partial_t\partial_x\mathcal{R}\wpp_{2q-1}\|_{L^\infty_tL^\infty}+\ell^2_{2q-1}\|\partial_t\mathcal{R}\wpp_{2q-1}\|_{L^\infty_tL^\infty}\|\partial_x\mathcal{R}\wpp_{2q-1}\|_{L^\infty_tL^\infty})\\
 \lesssim& C^{5/2}_0\ell^2_{2q-1}\lambda^{3}_{2q-1}.
\end{align*}}
Using this estimate, together with  Proposition \ref{Est-fe}, $k\perp\kp$ and \eqref{Rwj}, we bound $E^{(2)}_{2q}$ in \eqref{def-E2} as
\begin{align}
    \|E^{(2)}_{2q}\|_{L^1_t \dot{B}^{-1}_{\infty,\infty}}
    &\lesssim  \bigl\| \chi^2_{2q-1}\eta^2_{2q-1} \bigl[
        a^2_{(k,q)}-2000C_0a^2_k\big({\rm Id}+\tfrac{ \mathcal{R}\wpp_{2q-1}
 }{  1000C_0}\big) \bigr] \bigr\|_{L^\infty_{t} L^\infty}\notag \\
    &\quad +C_0(\lambda^{-r}_{2q}\|\chi^2_{2q-1}\eta^2_{2q-1}a_k\big({\rm Id}+\tfrac{\mathcal{R}\wpp_{2q-1}}{1000C_0}\big)\|_{L^\infty_tC^1}\notag\\
    &\quad+\lambda^{-2r}_{2q} \|\chi^2_{2q-1}\eta^2_{2q-1}a_k\big({\rm Id}+\tfrac{\mathcal{R}\wpp_{2q-1}}{1000C_0}\big)\|_{L^\infty_tC^2}\notag\\
    &\quad +\lambda^{-3r}_{2q}\|\chi^2_{2q-1}\eta^2_{2q-1}a_k\big({\rm Id}+\tfrac{\mathcal{R}\wpp_{2q-1}}{1000C_0}\big)\|_{L^\infty_tC^3})\notag\\
    &\lesssim C^{5/2}_0 \ell^2_{2q-1} \lambda^{3}_{2q-1}+C_0\lambda^{-r}_{2q}\lambda_{2q-1}\lesssim C^{5/2}_0 \ell^2_{2q-1} \lambda^{3}_{2q-1},\label{E2-B-1}
\end{align}
\begin{align}
    \|E^{(2)}_{2q}\|_{L^1_t \dot{B}^{0}_{\infty,\infty}}
    \lesssim&  \bigl\| \chi^2_{2q-1}\eta^2_{2q-1} 
          \bigl[
        a^2_{(k,q)}-2000C_0a^2_k\big({\rm Id}+\tfrac{ \mathcal{R}\wpp_{2q-1}
 }{  1000C_0}\big) \bigr]\bigr\|_{L^\infty_{t} C^1}\notag \\
    &+C_0\|\chi^2_{2q-1}\eta^2_{2q-1}a_k\big({\rm Id}+\tfrac{\mathcal{R}\wpp_{2q-1}}{1000C_0}\big)\|_{L^\infty_tC^1}
    \lesssim C_0 \lambda_{2q-1}\label{E2-B0}
\end{align}
and 
\begin{align}\label{E2-B1}
    \|E^{(2)}_{2q}\|_{L^1_t \dot{B}^{1}_{\infty,\infty}}
    \lesssim& C_0 \bigl\| \chi^2_{2q-1}\eta^2_{2q-1} a^2_k
          \bigl( \mathrm{Id} + \tfrac{\mathcal{R}\wpp_{2q-1}}{1000C_0} \bigr)
          \bigr\|_{L^\infty_{t} C^2}\notag\\
          &+ C_0 \lambda^{r}_{2q}
          \bigl\| \chi^2_{2q-1}\eta^2_{2q-1} a_k
          \bigl( \mathrm{Id} + \tfrac{\mathcal{R}\wpp_{2q-1}}{1000C_0} \bigr)
          \bigr\|_{L^\infty_{t}C^1} \notag\\
    \lesssim& C^3_0\lambda^2_{2q-1} +C^2_0\lambda^{r}_{2q}\lambda_{2q-1}\lesssim C^2_0\lambda^{r}_{2q}\lambda_{2q-1}.
\end{align}
For $E^{(3)}_{2q}$ in \eqref{def-E3}, we deduce from \eqref{Rwj} and \eqref{E-wpm-wpr} that
\begin{align}
    \|E^{(3)}_{2q}\|_{L^1_t \dot{B}^{-1}_{\infty,\infty}
   }\lesssim&\|\partial_t(\Div \mathcal R \wpp_{2q-1} )e^{-2\lambda^2_{2q}t}\|_{L^1_t \dot{B}^{-1}_{\infty,\infty}}+\|(\wpm_{2q},\wpr_{2q})\|_{L^1_tL^\infty}\|\wpr_{2q}\|_{L^\infty_tL^\infty}\notag\\
    \lesssim &\|e^{-2\lambda^2_{2q}t}\|_{L^1_t}\|\partial_t \mathcal R \wpp_{2q-1} \|_{L^\infty_t L^\infty}+C^{3/2}_0\lambda^{-1+r}_{2q}\notag\\
    \lesssim& C^{3/4}_0\lambda^{-2}_{2q}\lambda^{2}_{2q-1}+C^{3/2}_0\lambda^{-1+r}_{2q}\lesssim C^{3/2}_0\lambda^{-1+r}_{2q}\label{E3-B-1}
\end{align}
and 
\begin{align}
    \|E^{(3)}_{2q}\|_{L^1_t \dot{B}^{1}_{\infty,\infty}}\lesssim&\|\partial_t(\Div \mathcal R \wpp_{2q-1} )e^{-2\lambda^2_{2q}t}\|_{L^1_t \dot{B}^{1}_{\infty,\infty}}\notag\\
&+\|(\wpm_{2q},\wpr_{2q})\|_{L^1_tL^\infty}\|\wpr_{2q}\|_{L^\infty_tC^2}+\|(\wpm_{2q},\wpr_{2q})\|_{L^1_tC^2}\|\wpr_{2q}\|_{L^\infty_{t,x}}\notag\\
    \lesssim &\|e^{-2\lambda^2_{2q}t}\|_{L^1_t}\|\partial_t \mathcal R \wpp_{2q-1} \|_{L^\infty_t C^2}+C^{3/2}_0\lambda^{1+r}_{2q}\notag\\
    \lesssim & C^{3/4}_0\lambda^{-2}_{2q}\lambda^{4}_{2q-1}+C^{3/2}_0\lambda^{1+r}_{2q}\lesssim C^{3/2}_0\lambda^{1+r}_{2q}.\label{E3-B1}
\end{align}
Collecting \eqref{E1-B-1}-- \eqref{E3-B1} together shows
\begin{align}
    &\|E^{(1)}_{2q}\|_{L^1_t \dot B^{0}_{\infty,1}} \lesssim \|E^{(1)}_{2q}\|^{1/2}_{L^1_t \dot B^{-1}_{\infty,\infty}} \|E^{(1)}_{2q}\|^{1/2}_{L^1_t \dot B^{1}_{\infty,\infty}} \lesssim C_0\lambda^r_{2q},\label{E1-B0}\\
       &\|E^{(2)}_{2q}\|_{L^1_t \dot B^{0}_{\infty,1}} \lesssim \|E^{(2)}_{2q}\|^{1/2}_{L^1_t \dot B^{-1}_{\infty,\infty}} \|E^{(2)}_{2q}\|^{1/2}_{L^1_t \dot B^{1}_{\infty,\infty}} \lesssim C_0\ell_q\lambda^2_{2q-1}\lambda^r_{2q},\label{E2-B01}\\
       &\|E^{(3)}_{2q}\|_{L^1_t \dot B^{0}_{\infty,1}} \lesssim \|E^{(3)}_{2q}\|^{1/2}_{L^1_t \dot B^{-1}_{\infty,\infty}} \|E^{(3)}_{2q}\|^{1/2}_{L^1_t \dot B^{1}_{\infty,\infty}} \lesssim C^2_0\lambda^r_{2q}.\label{E3-B0}
\end{align}
Using \eqref{E1-B-1}, \eqref{E2-B0}, \eqref{E3-B-1}, \eqref{E1-B0} and \eqref{E3-B0}, one has
\begin{align*}
 \|E_{2q}\|_{L^1_t \dot B^{-1+\frac{2}{p}}_{\infty,1}} \lesssim &\sum_{i=1}^3\|E^{(i)}_{2q}\|^{1-\frac{2}{p}}_{L^1_t \dot B^{-1}_{\infty,\infty}} \|E^{(i)}_{2q}\|^{\frac{2}{p}}_{L^1_t \dot B^{0}_{\infty,\infty}}\\
 \lesssim &C^2_0\lambda^{-1+\frac{2}{p}+r}_{2q}+C^3_0\ell^{2(1-\frac{2}{p})}_{2q-1}\lambda^{3-\frac{4}{p}}_{2q-1}\lesssim C^3_0\lambda^{-90}_{2q-1}
\end{align*}
where the last inequality holds for $\ell_{2q-1}$ defined by \eqref{def-ell}. Since  $E_{2q}$ has compact support and $2<p<\infty$,  we have
\begin{align}\label{E-Bp}
  \|E_{2q}\|_{L^1_t\dot{B}^{ -1+\frac{2}{p}}_{p,1}}
  \lesssim&\|E_{2q}\|_{L^1_t\dot B^{-1+\frac{2}{p}}_{\infty ,1}}\lesssim  C^3_0\lambda^{-90}_{2q-1}.
\end{align}
and
 \begin{align}
  \|E_{2q}\|_{L^1_t\dot{B}^{ -\frac{2}{p}}_{2,1}}
  \lesssim&\|E_{2q}\|_{L^1_t\dot B^{-\frac{2}{p}}_{\frac{2p}{p-2},1}}\lesssim  \|E_{2q}\|_{L^1_t\dot B^{-\frac{2}{p}}_{\infty,1}}
  \lesssim\sum_{i=1}^3\|E^{(i)}_{2q}\|^{\frac{2}{p}}_{L^1_t\dot B^{-1}_{\infty,\infty}}\|E^{(i)}_{2q}\|^{1-\frac{2}{p}}_{L^1_t\dot B^{0}_{\infty,\infty}}\notag\\
  \lesssim&C^2\lambda^{-\frac{2}{p}+r}_{2q}+C^3_0\ell^{\frac{4}{p}}_{2q-1}\lambda^{1+\frac{4}{p}}_{2q-1}\lesssim C^3_0\lambda^{-90}_{2q-1}.\label{E-B2}
\end{align}
With the aid of \eqref{E-Bp} and $E_{2q}$ has compact support, we have
\begin{align}\label{E-B_1}
 \|E_{2q}\|_{L^1_t B^{-1}_{1, 1}} \lesssim  \|E_{2q}\|_{L^1_t \dot B^{-1}_{\infty, 1}}\lesssim C^3_0\lambda^{-90}_{2q-1}.
\end{align}
Hence, collecting \eqref{E1-B0}--\eqref{E-B_1}, we  finish the proof of Proposition \ref{E2q}.
\end{proof}
\begin{proposition}[Estimates for $F_{2q}$]\label{Pro-F}Let $F_{2q}$ be given in \eqref{def-F}. For the fixed $2<p<\infty$ given in Proposition~\ref{iteration}, we have
\begin{align}\label{es-F}
 \|F_{2q}\|_{L^1_t (\dot B^{-1+\frac{2}{p}}_{p,1}\cap \dot B^{-\frac{2}{p}}_{2,1}\cap B^{-1}_{1,1})}\lesssim C^2_0(\lambda^{-1+\frac{2}{p}+r}_{2q}+\lambda^{-\frac{2}{p}+r}_{2q}) ,  \quad \|F_{2q}\|_{L^1_t \dot B^{0}_{\infty,1}} \lesssim C^2_0\lambda^r_{2q}  .
\end{align}
\end{proposition}
\begin{proof}
With the aid of  \eqref{u-Linfty} and Proposition \ref{Est-wpws}, we have
\begin{align}\label{F1}
&\Big \|F_{2q}-\big((\partial_t  \wpp_{2q} -\Delta \wpp_{2q})-\Delta \wss_{2q} \big)\Big\|_{L^1_t\dot{B}^{-1}_{\infty,\infty}}\notag\\
\lesssim&\|\wpp_{2q}\|_{ L^1_{t}L^\infty}
(\|(\wss_{2q},u_{2q-2})\|_{L^\infty_{t}L^\infty }+\|\wss_{2q}\|_{L^2_{t}L^\infty}(\|(\wss_{2q},u_{2q-2})\|_{L^2_{t}L^\infty})\notag\\
\lesssim&C^2_0\lambda^{-1}_{2q}\lambda_{2q-1},
\end{align}
and
\begin{align}
&\Big \|F_{2q}-\big((\partial_t  \wpp_{2q} -\Delta \wpp_{2q})-\Delta \wss_{2q} \big)\Big\|_{L^1_t\dot{B}^{0}_{\infty,1}}\notag\\
\lesssim&\|(\wpp_{2q},\wss_{2q})\|_{ L^2_{t}L^\infty}
\|(\wss_{2q},u_{2q-2})\|_{L^2_{t}\dot B^{1}_{\infty,1}}+\|(\wpp_{2q},\wss_{2q})\|_{ L^1_{t}\dot B^{1}_{\infty,1}}
\|(\wss_{2q},u_{2q-2})\|_{L^\infty_{t,x}}\notag\\
\lesssim&C^2_0\lambda_{2q-2}+C_0^2\lambda^{-1}_{2q}\lambda_{2q-1}\lesssim C^2_0\lambda_{2q-2}.\label{F2-re}
\end{align}    
One easily deduce from \eqref{wp-wsL1} that
\begin{align}
    \|\Delta \wss_{2q}\|_{L^1_t\dot{B}^{-1}_{\infty,\infty}}\lesssim& \| \wss_{2q}\|_{L^1_t\dot{B}^{1}_{\infty,\infty}}
\le C^{3/4}_0\lambda^{-2}_{2q}\lambda^2_{2q-1},
\end{align}
and 
\begin{align}
    \|\Delta \wss_{2q}\|_{L^1_t\dot{B}^{0}_{\infty,1}}\lesssim& \| \wss_{2q}\|_{L^1_t\dot{B}^{2}_{\infty,1}}\lesssim \|\wpp_{2q-1}\|_{L^1_t \dot B^{2}_{\infty,1}}
\le C^{3/4}_0\lambda^{-2}_{2q}\lambda^3_{2q-1}.
\end{align}
Using \eqref{dec-wh}, we rewrite
\begin{align*}
  \partial_t  \wpp_{2q} -\Delta \wpp_{2q}= (\partial_t - \Delta)\wpm_{2q}+ (\partial_t - \Delta)\wpr_{2q}.
\end{align*}
Note that $ (\partial_t - \Delta)(e^{-\lambda^2_{2q}t} e^{\ii\lambda_{2q}k\cdot x})=0$, one infers from \eqref{def-wpm} that
\begin{align*}
 (\partial_t - \Delta)\wpm_{2q}
=& \lambda_{2q }e^{-\lambda^2_{2q }t}\sum_{k\in\Lambda}\big((\partial_t - \Delta)(a_{(k,2q)} \chi_{2q-1}\eta_{2q-1}\phi_{k,2q} ) \big)e^{\ii\lambda_{2q }k\cdot x} \kp\\
&+\ii\lambda^2_{2q }e^{-\lambda^2_{2q }t}\sum_{j=1}^2\sum_{k\in\Lambda}\partial_{x_j}(a_{(k,q)}\chi_{2q-1}\eta_{2q-1}\phi_{k,2q})
e^{\ii\lambda_{2q }k\cdot x}k_j \kp
\end{align*}
Applying Proposition \ref{Est-fe} to the above equality yields
\begin{align*}
& \| (\partial_t - \Delta)\wpm_{2q} \|_{{L}^1_t \dot{B}^{ -1}_{\infty,\infty}}\notag\\
\lesssim&\|\lambda_{2q}e^{-\lambda^2_{2q}t}\|_{L^1_t} \sum_{k\in\Lambda}\big(\lambda^{-1}_{2q} \|   (\partial_t - \Delta)(a_{(k,q)} \chi_{2q-1}\eta_{2q-1}\phi_{k,2q}) \|_{{L}^{\infty}_{t,x}}\notag\\
&+ \lambda^{-2}_{2q} \|   (\partial_t - \Delta)(a_{(k,q)} \chi_{2q-1}\eta_{2q-1}\phi_{k,2q}) \|_{{L}^{\infty}_{t}C^1}+\lambda^{-3}_{2q} \|   (\partial_t - \Delta)(a_{(k,q)} \chi_{2q-1}\eta_{2q-1}\phi_{k,2q}) \|_{{L}^{\infty}_{t}C^3}\big)\notag\\
&+ \|\lambda^2_{2q}e^{-\lambda^2_{2q}t}\|_{L^1_t}\sum_{j=1}^2\sum_{k\in\Lambda}\big(\lambda^{-1}_{2q}\|  \partial_{x_j}(a_{(k,q)}\chi_{2q-1}\eta_{2q-1}\phi_{k,2q})\|_{{L}^{\infty}_{t,x}}\notag\\
&+ \lambda^{-2}_{2q} \|  \partial_{x_j}(a_{(k,q)}\chi_{2q-1}\eta_{2q-1}\phi_{k,2q})\|_{{L}^{\infty}_{t}C^1}+\lambda^{-3}_{2q} \|  \partial_{x_j}(a_{(k,q)}\chi_{2q-1}\eta_{2q-1}\phi_{k,2q})\|_{{L}^{\infty}_{t}C^3}\big)\notag\\
\lesssim &C_0\lambda^{-2+2r}_{2q}\ell^{-1}_{2q-1}+C_0\lambda^{-1+r}_{2q}
\lesssim C_0\lambda^{-1+r}_{2q},
\end{align*}
and
\begin{align*}
 &\| (\partial_t - \Delta)\wpm_{2q} \|_{{L}^1_t \dot{B}^{1}_{\infty,\infty}}\\
 \lesssim&\|\lambda_{2q}e^{-\lambda^2_{2q}t}\|_{L^1_t}\sum_{k\in\Lambda}\big(\|   (\partial_t - \Delta)(a_{(k,q)} \chi_{2q-1}\eta_{2q-1}\phi_{k,2q}) \|_{{L}^{\infty}_{t}C^1}+ \lambda_{2q} \|   (\partial_t - \Delta)(a_{(k,q)} \chi_{2q-1}\eta_{2q-1}\phi_{k,2q}) \|_{{L}^{\infty}_{t,x}}\big)\\
 &+\|\lambda^2_{2q}e^{-\lambda^2_{2q}t}\|_{L^1_t}\sum_{j=1}^2\sum_{k\in\Lambda}(\|  \partial_{x_j}(a_{(k,q)}\chi_{2q-1}\eta_{2q-1}\phi_{k,2q})\|_{{L}^{\infty}_{t}C^1}+ \lambda_{2q}\|  \partial_{x_j}(a_{(k,q)}\chi_{2q-1}\eta_{2q-1}\phi_{k,2q})\|_{{L}^{\infty}_{t,x}})\\
 \lesssim&C_0\lambda^{1+r}_{2q}.
\end{align*}
Hence, we obtain from the two above estimates that
\begin{align*}
 \| (\partial_t - \Delta)\wpm_{2q} \|_{{L}^1_t \dot{B}^{ 0}_{\infty,1}}\lesssim \| (\partial_t - \Delta)\wpm_{2q} \|^{1/2}_{{L}^1_t \dot{B}^{ -1}_{\infty,\infty}}\| (\partial_t - \Delta)\wpm_{2q} \|^{1/2}_{{L}^1_t \dot{B}^{ 1}_{\infty,\infty}}\lesssim C_0\lambda^r_{2q}.
\end{align*}
By the same argument for estimating $\wpr_{2q}$, one infers that
\begin{align}\label{F3}
\| (\partial_t - \Delta)\wpp_{2q} \|_{{L}^1_t \dot{B}^{ -1}_{\infty,\infty}}\lesssim C_0\lambda^{-1+r}_{2q},\quad \| (\partial_t - \Delta)\wpp_{2q} \|_{{L}^1_t \dot{B}^{ 0}_{\infty,1}}\lesssim C_0\lambda^{r}_{2q}.
\end{align}
Hence, we have
\begin{align*}
 \|F_{2q}\|_{L^1_t \dot B^{-1}_{\infty,\infty}} \le C^2_0\lambda^{-1+r}_{2q} ,  \quad \|F_{2q}\|_{L^1_t \dot B^{0}_{\infty,1}} \le C^2_0\lambda^r_{2q}.   
\end{align*}
Note that $F_{2q}$ has compact support, we immediately have
\begin{align*}
 \|F_{2q}\|_{L^1_t(\dot B^{-1+\frac{2}{p}}_{p,1}\cap \dot B^{-\frac{2}{p}}_{2,1}\cap B^{-1}_{1,1})}\lesssim    \|F_{2q}\|_{L^1_t(\dot B^{-1+\frac{2}{p}}_{\infty,1}\cap \dot B^{-\frac{2}{p}}_{\infty,1}\cap \dot B^{-1}_{\infty,1})} \lesssim C^2_0(\lambda^{-1+\frac{2}{p}+r}_{2q}+\lambda^{-\frac{2}{p}+r}_{2q})
\end{align*}
and complete the proof of Proposition \ref{Pro-F}.
\end{proof}
\begin{proposition}[Estimates for $G_{2q}$]\label{G2q}Let $G_{2q}$ be given in \eqref{def-G}. For the fixed $2<p<\infty$ given in Proposition~\ref{iteration}, we have
\begin{align}\label{es-F}
 \|G_{2q}\|_{L^1_t(\dot B^{-1+\frac{2}{p}}_{p,1}\cap \dot B^{-\frac{2}{p}}_{2,1}\cap B^{-1}_{1,1})} \lesssim C^2_0\lambda_{2q-2}(\lambda^{-1+\frac{2}{p}}_{2q}+\lambda^{-\frac{2}{p}}_{2q}),  \quad \|G_{2q}\|_{L^1_t \dot B^{0}_{\infty,1}} \lesssim  C^2_0\lambda_{2q-2}.
\end{align}
\end{proposition}
 \begin{proof}
  By \eqref{u-Linfty} and Proposition \ref{Est-wpws}, we have
  \begin{align*}
     \|G_{2q}\|_{L^1_t\dot B^{-1}_{\infty,\infty}}\lesssim  \|(\wpp_{2q}+\wss_{2q})\otimes B_{2q-2}\|_{L^1_t \dot{B}^{0}_{\infty, \infty}}\lesssim \|(\wpp_{2q},\wss_{2q})\|_{L^1_tL^\infty }\|B_{2q-2}\|_{L^\infty_{t,x}}\lesssim C^2_0\lambda^{-1}_{2q}\lambda_{2q-2}
 \end{align*}
 and
 \begin{align*}
     \|G_{2q}\|_{L^1_t\dot B^{0}_{\infty,1}}\lesssim&  \|(\wpp_{2q}+\wss_{2q})\otimes B_{2q-2}\|_{L^1_t \dot{B}^{1}_{\infty, 1}}\\
     \lesssim& \|(\wpp_{2q},\wss_{2q})\|_{L^2_tL^\infty }\|B_{2q-2}\|_{L^2_{t}\dot B^{1}_{\infty,1}}+\|(\wpp_{2q},\wss_{2q})\|_{L^1_t\dot B^{1}_{\infty,1} }\|B_{2q-2}\|_{L^\infty_{t,x}}\\
     \lesssim& C^2_0\lambda_{2q-2}.
 \end{align*}
Since $G_{2q}$ has compact support, one obtains from the above two estimates that
\begin{align*}
 \|G_{2q}\|_{L^1_t(\dot B^{-1+\frac{2}{p}}_{p,1}\cap \dot B^{-\frac{2}{p}}_{2,1}\cap B^{-1}_{1,1})}\lesssim    \|G_{2q}\|_{L^1_t(\dot B^{-1+\frac{2}{p}}_{\infty,1}\cap \dot B^{-\frac{2}{p}}_{\infty,1}\cap \dot B^{-1}_{\infty,1})} \lesssim C^2_0\lambda_{2q-2}(\lambda^{-1+\frac{2}{p}}_{2q}+\lambda^{-\frac{2}{p}}_{2q})
\end{align*}
and thus complete the proof of Proposition \ref{G2q}.
 \end{proof}
\subsection{Estimates of $\wmhd_{2q}$ and  $\dmhd_{2q}$}\label{E-wmhd} Based on the above estimates for the forcing terms $E_{2q}, F_{2q}, G_{2q}$, we now formulate the global-in-time existence of solutions to the system \eqref{e:wt} as follows.
\begin{proposition}\label{wtq-H3}
 For the fixed $2<p<\infty$ given in Proposition~\ref{iteration}, there exists a global-in-time solution $(\wmhd_{2q},\dmhd_{2q}) \in C^\infty_{t,x}(\mathbb{R}^{+}\times\mathbb{R}^2)$ of the system \eqref{e:wt}
satisfying, 
 \begin{align}
&\|(\wmhd_{2q},\dmhd_{2q})\|_{\widetilde L^{\infty}_t\bigl(\dot B^{-1+\frac{2}{p}}_{p,1}\cap\dot B^{-\frac{2}{p}}_{2,1}\bigr)\cap
L^{1}_t\bigl(\dot B^{1+\frac{2}{p}}_{p,1}\cap\dot B^{2-\frac{2}{p}}_{2,1}\bigr)} \le \lambda^{-20}_{2q-1}, \label{wtq12}
\end{align}
and
\begin{align}\label{wd-B0}
\|(\wmhd_{2q},\dmhd_{2q})\|_{\widetilde L^{\infty}_t\dot{B}^{0}_{\infty,1}\cap L^{1}_t\dot{B}^{2}_{\infty,1}} \le \lambda^{2r}_{2q}. 
\end{align}
Moreover, we have
\begin{align}\label{E-B-1}
\|(\wmhd_{2q},\dmhd_{2q})\|_{L^{\infty}([0,T];{B}^{-1}_{1,1})\cap L^{1}([0,T];{B}^{1}_{1,1})}\le \lambda^{-15}_{2q-1}.
\end{align}
\end{proposition}
\begin{proof}
    Since the initial data are smooth, local well-posedness for \eqref{e:wt} follows directly from standard theory.  Therefore, it suffices to prove the global-in-time existence.
    
    Let $a_{2q}:=\wmhd_{2q}+\dmhd_{2q}$ and $c_{2q}:=\wmhd_{2q}-\dmhd_{2q}$. The system \eqref{e:wt} reduces to
\begin{equation}
\left\{ \begin{alignedat}{-1}
&\partial_t a_{2q}-\Delta a_{2q}+ (c_{2q}+u_{2q-2}+\wpp_{2q}+\wss_{2q}-B_{2q-2})\cdot\nabla a_{2q} +\nabla p^{(\textup{m})}_{2q } \\
 &\qquad= -c_{2q }\cdot\nabla  (u_{2q-2}+\wpp_{2q}+\wss_{2q}-B_{2q-2})   -E_{2q }-F_{2q }-G_{2q},
\\
&\partial_t c_{2q}-\Delta c_{2q}+  (a_{2q}+u_{2q-2}+\wpp_{2q}+\wss_{2q}+B_{2q-2})\cdot\nabla c_{2q} +\nabla p^{(\textup{m})}_{2q }\\
&\qquad=-a_{2q }\cdot\nabla(u_{2q-2}+\wpp_{2q}+\wss_{2q}-B_{2q-2})-E_{2q }-F_{2q }+G_{2q} ,\\
& \Div a_{2q}=\Div c_{2q}= 0,
  \\
  & (a_{2q},c_{2q}) |_{t=0}=  (\lambda^{-50}_{2q-1}f(x),~-\lambda^{-50}_{2q-1}f(x)).
\end{alignedat}\right.
\label{e:acwt}
\end{equation}
For any $T>0$, we define 
\[
\|(a_{2q},c_{2q})\|_{E_T} := 
\|(a_{2q},c_{2q})\|_{\widetilde L^\infty_T\bigl(\dot B^{-1+\frac{2}{p}}_{p,1}\cap\dot B^{-\frac{2}{p}}_{2,1}\bigr)\cap
L^1_T\bigl(\dot B^{1+\frac{2}{p}}_{p,1}\cap\dot B^{2-\frac{2}{p}}_{2,1}\bigr)}
\]
and
\[
T^{\star} := \sup\Bigl\{ T\ge0 \;\big|\; \|(a_{2q},c_{2q})\|_{E_T}\le \lambda^{-20}_{2q-1} \Bigr\}.
\]
Using Proposition~\ref{T-D}, due to 
$$\Div\, u_{2q-2}=\Div\, B_{2q-2}=\Div\, \wpp_{2q}=\Div\, \wss_{2q}=0,$$
we obtain, for every $T\le T^{\star}$,
\begin{align}
\|(a_{2q},c_{2q})\|_{E_T}\le &
C\exp\Bigl(C\int_0^T\|\nabla(a_{2q},c_{2q},\wss_{2q},\wpp_{2q},u_{2q-2},B_{2q-2})(s)\|_{L^\infty}\,\mathrm{d}s\Bigr)\nonumber\\
&\times\Bigl(\lambda^{-50}_{2q-1}\|f\|_{\dot B^{-1+\frac{2}{p}}_{p,1}\cap\dot B^{-\frac{2}{p}}_{2,1}}
+\int_0^T\|(E_{2q}, F_{2q}, G_{2q})(s)\|_{\dot B^{-1+\frac{2}{p}}_{p,1}\cap\dot B^{-\frac{2}{p}}_{2,1}}\mathrm{d}s\nonumber\\
&\quad+\int_0^T\|\nabla(\wss_{2q},\wpp_{2q},u_{2q-2},B_{2q-2})(s)\|_{L^\infty}
\|(a_{2q},c_{2q})(s)\|_{\dot B^{-1+\frac{2}{p}}_{p,1}\cap\dot B^{-\frac{2}{p}}_{2,1}}\,\mathrm{d}s\Bigr). \label{estE}
\end{align}
From Propositions~\ref{E2q}--\ref{G2q}, we have, for any $T>0$,
\begin{align}\label{E-F-G}
\int_0^T\|(E_{2q}, F_{2q},G_{2q})(s)\|_{\dot B^{-1+\frac{2}{p}}_{p,1}\cap\dot B^{-\frac{2}{p}}_{2,1}}\mathrm{d}s
\le C^3_0\lambda^{-30}_{2q-1}.    
\end{align}
Moreover, by the inductive estimate \eqref{u-L1C1} and Proposition \ref{Est-wpws}, we have
\begin{align}
 \int_0^T\|\nabla(\wss_{2q},\wpp_{2q},u_{2q-2},B_{2q-2})(s)\|_{L^\infty}\,\mathrm{d}s \le C_0(q+1). \label{boundLinf}   
\end{align}
For any $T\le T^{\star}$, substituting 
\eqref{E-F-G} and \eqref{boundLinf}  into \eqref{estE}  yields that
\begin{align*}
\|(a_{2q},c_{2q})\|_{E_T}
\le & C\exp\big(CC_0(q+1)\big)\Big(2C^3_0\lambda^{-30}_{2q-1}\\
&+\int_0^T\|\nabla (\wss_{2q}, \wpp_{2q}, u_{2q-2},b_{2q-2})(s)\|_{L^\infty}\|(a_{2q},c_{2q})(s)\|_{\dot B^{-1+\frac{2}{p}}_{p,1}\cap\dot B^{-\frac{2}{p}}_{2,1}}\dd s\Big).    
\end{align*}
By Gronwall's inequality and \eqref{boundLinf}, we have
\begin{align*}
  &\|(a_{2q},c_{2q})\|_{E_T}\\
\le&  2\lambda^{-30}_{2q-1}CC^3_0\exp(CC_0(q+1)\big)\exp\Big(C\exp\big(CC_0(q+1))\int_0^T\|\nabla (\wss_{2q}, \wpp_{2q}, u_{2q-2},b_{2q-2})(s)\|_{L^\infty}\Big)\\
\le&2\lambda^{-30}_{2q-1}CC^3_0\exp(CC_0(q+1)\big)\exp\Big(CC_0(q+1)\exp\big(CC_0(q+1))\Big)\\
\le&2\lambda^{-30}_{2q-1}  \exp(2\exp(2CC^3_0(q+1))) \le \lambda^{-25}_{2q-1},
\end{align*}
where the last inequality holds by choosing $b$ large enough such that
$b\ge 10e^{10CC^3_0}$ and $a>2e$.
A standard continuity argument then shows $T^{\star}=\infty$, and we actually have
\begin{align}\label{E-ac--1}
\|(a_{2q},c_{2q})\|_{_{\widetilde L^\infty_t\bigl(\dot B^{-1+\frac{2}{p}}_{p,1}\cap\dot B^{-\frac{2}{p}}_{2,1}\bigr)\cap
L^1_t\bigl(\dot B^{1+\frac{2}{p}}_{p,1}\cap\dot B^{2-\frac{2}{p}}_{2,1}\bigr)}}\le \lambda^{-20}_{2q-1}.
\end{align}
Note that $\|\lambda^{-50}_{2q-1} f\|_{\dot B^0_{\infty,1}\cap B^{-1}_{1,1}}\le \lambda^{-40}_{2q-1}$ and by Propositions \ref{E2q}--\ref{G2q}, we have 
\begin{align*}
 \|(E_{2q}, F_{2q}, G_{2q})\|_{L^1_t \dot B^{0}_{\infty,1}}\le C^3_0 \lambda^r_{2q},\,\,  \|(E_{2q}, F_{2q}, G_{2q})\|_{L^1_t B^{-1}_{1,1}}\le C^3_0 \lambda^{-30}_{2q-1}
\end{align*}
 Combining these estimates with the inductive bounds \eqref{u-Linfty}, \eqref{u-L1C1} and \eqref{E-ac--1}, and following  in the same argument as above, we obtain \eqref{wd-B0} and  \eqref{E-B-1}. Hence, we complete the proof of Proposition \ref{wtq-H3}.
\end{proof}

\subsection{Verifying the $2q-$th iteration}Now let us finish the proof of Proposition \ref{iteration}. With the aid of \eqref{def-wh}, Propositions \ref{Est-wpws},  \ref{cedu} and \ref{wtq-H3}, one immediately deduces that $\wpp_{2q}$, $\wss_{2q}$, $\wmhd_{2q}$ and $\dmhd_{2q}$ satisfy \eqref{def-R}--\eqref{wm-C} at $2q$-level.
  Let
   \begin{align*}
       w_{2q}=\wpp_{2q}+\wss_{2q}+\wmhd_{2q},
   \end{align*}
and
   \begin{align*}
    u_{2q}=u_{2q-2}+w_{2q},\quad B_{2q}=B_{2q-2}+\dmhd_{2q}.
\end{align*}
We immediately show \eqref{u-decomposition} and \eqref{b-decomposition} for $j=2q$. Thanks to Propositions \ref{Est-wpws} and \ref{wtq-H3}, we have
\begin{align*}
&\|(w_{2q}, \dmhd_{2q})\|_{L^2_TL^2}\le \lambda^{-15}_{2q-1}\le 2C^{3/4}_02^{-\frac{2q-1}{2}},\\
& \|(w_{2q}, \dmhd_{2q})\|_{L^\infty_t 
\dot B^{0}_{\infty,1}\cap L^1_t\dot B^2_{\infty,1}}\le    2   C^{3/4}_0 \lambda_{2q}+\lambda^{2r}_{2q}\le \frac{1}{2}C_0\lambda_{2q},\\
&\|(w_{2q}, \dmhd_{2q})\|_{L^1_t 
\dot B^{1}_{\infty,1}}\le  2C^{3/4}_0+ \lambda ^{-20}_{2q-1}\le C_0.
\end{align*}
 The above two estimates combined with \eqref{u-Linfty} --\eqref{u-L1C1} for $j=2q-1$ show that \eqref{u-Linfty} and \eqref{u-L1C1}  hold for $j=2q$.  Hence, we complete the proof of Proposition \ref{iteration} for the case where $m$ is even. For the case that $m$ is odd, we have $m=2q+1$ for some $q \ge 1$. We construct $\wpp_{2q+1}, \wss_{2q+1}$, $ \wmhd_{2q+1}$ and $\dmhd_{2q+1}$ by replacing the parameter $2q$ in definitions \eqref{def-wh}, \eqref{def-ws} and in the system \eqref{e:wt} with $2q+1$. Let $w_{2q+1}=\wpp_{2q+1}+\wss_{2q+1}+\wmhd_{2q+1}$, $u_{2q+1}=u_{2q-1}+w_{2q+1}$ and $B_{2q+1}=B_{2q-1}+\dmhd_{2q+1}$.  Following the same argument as for the even case, Proposition \ref{iteration} is  established also for odd $m$, which finishes the proof of Proposition \ref{iteration}.
\section*{Acknowledgement}
We thank Professor Mimi Dai for the valuable comments on this manuscript. This work was supported by the National Key Research and Development Program of China    (Grant number No. 2022YFA1005700) and  the National Natural Science Foundation of China (Grant numbers No.12371095, No. 12401277 and No.
12501304).

\end{document}